\documentclass[a4paper, reqno,12pt]{amsart}
\usepackage{tikz, tikz-3dplot}
\usepackage{amsmath, amsthm, amssymb}

\theoremstyle{plain}
\newtheorem{te}{Theorem}[section]
\newtheorem{lem}[te]{Lemma}
\newtheorem{co}[te]{Corollary}
\newtheorem{pr}[te]{Proposition}
\newtheorem{de}[te]{Definition}
\newtheorem{ex}[te]{Example}

\newtheorem{con}[te]{Conjecture}
\theoremstyle{remark}
\newtheorem{re}[te]{Remark}
\newtheorem*{ack*}{Acknowledgment}

\def\0{{\bf 0}}

\def\K{{\mathcal K}}
\def\T{{\mathbb T}}
\def\R{{\mathbb R}}

\def\Z{{\mathbb Z}}

\def\gcd{{\operatorname{gcd}}}

\def\nint{\mathop{\diagup\kern-13.0pt\int}}

\def\les{{\;\lessapprox}\;}\def\ges{{\;\gtrapprox}\;}

\def\Nc{{\mathcal N}}
\def\Mc{{\mathcal M}}
\def\Bc{{\mathcal B}}
\def\Ac{{\mathcal A}}
\def\Pc{{\mathcal P}}
\def\Sc{{\mathcal S}}\def\Vc{{\mathcal V}}
\def\Lc{{\mathcal L}}

\begin{document}
	\title[Level set estimates for the periodic Schr\"odinger maximal function]{Level set estimates for the periodic Schr\"odinger maximal function  on $\mathbb{T}^1$}
	\author{Ciprian Demeter}
	\address{Department of Mathematics, Indiana University, 831 East 3rd St., Bloomington IN 47405}
	\email{demeterc@iu.edu}

	\thanks{The author is  partially supported by the NSF grant  DMS-2349828}	
	
	\begin{abstract}
		We prove (essentially) sharp $L^4$ level set estimates for  the periodic Schr\"odinger maximal operator in a certain range of the cut-off parameter. 		
	\end{abstract}
	
	\maketitle
	\section{Introduction}

	We  use number theory and elementary graph theory  to prove the following level set estimate.
	\begin{te}
		\label{teorem1}
		Assume $\|a_n\|_{l^2}=1$ and $ N^{\frac14+\frac1{10}}\les \lambda\le N^{\frac12}$. Then	
		\begin{equation}
			\label{jcfoiepfcr[p,goo,g}
			|\{x\in[0,1]:\;\sup_{t}|\sum_{n=1}^Na_ne(nx+n^2t)|\ge \lambda\}|\les \frac{N}{\lambda^4}.
		\end{equation}
	\end{te}
	Our result covers a range of estimates in the following conjecture.
	\begin{con}
		\label{t1}
		Assume $\|a_n\|_{l^2}=1$ and $N^{\frac14}\le \lambda\le N^{\frac12}$. Then	
		$$|\{x\in[0,1]:\;\sup_{t}|\sum_{n=1}^Na_ne(nx+n^2t)|\ge \lambda\}|\les \frac{N}{\lambda^4}.$$
	\end{con}
	
	Apart from the $N^\epsilon$ losses hidden in the symbol $\les$, the conjectured level set estimates are sharp for all $\lambda$. Indeed, recall (see e.g. Proposition 13.4 in \cite{De}) that if $q\sim Q\ll N$ and
	$$x\in \bigcup_{b=0}^{q-1}\;[\frac{b}q-\frac{o(1)}{N},\frac{b}q+\frac{o(1)}{N}]\,$$
	then for each $q\sim Q$ we have
	$$|\sum_{n=1}^Ne(nx+\frac{n^2}q)|\sim \frac{N}{q^{1/2}}.$$
	When $Q=\frac{N}{\lambda^2}\ll N^{1/2}$, the intervals in $\bigcup_{q\sim Q}\bigcup_{1\le b\le q-1\atop{\gcd(b,q)=1}}\;[\frac{b}q-\frac{o(1)}{N},\frac{b}q+\frac{o(1)}{N}]$ are pairwise disjoint, so this set has measure $\sim \frac{Q^2}{N}= \frac{N}{\lambda^4}$. Moreover, letting $a_n=N^{-1/2}$, we find that for each $x$ in this set we have
	$$\sup_{t}|\sum_{n=1}^Na_ne(nx+n^2t)|\sim \frac{N^{1/2}}{Q^{1/2}}=\lambda.$$

	Conjecture \ref{t1} admits the following equivalent reformulation.
	\begin{con}
		\label{Maincon}	
		Assume $\|a_n\|_{l^2}=1$. Then for each $p\le 4$
		$$\|\sup_{t}|\sum_{n=1}^Na_ne(nx+n^2t)|\|_{L^p([0,1],dx)}\les N^{1/4}.$$
	\end{con}
	This is an estimate about the solution of the free Schr\"odinger equation on the torus. If proved, it would imply that, under enough Sobolev regularity, the solution converges to the initial data almost everywhere as $t\to 0$. The Euclidean version of this problem has been solved by Carleson \cite{Ca}. There are higher dimensional versions of Conjecture \ref{t1} that will not be investigated here. The more recent papers \cite{DGL}, \cite{DZ} solved the Euclidean version in all dimensions.
	
	The best upper bound in the literature for Conjecture \ref{Maincon} is $\les N^{\frac13}$, in the range $1\le p\le 6$, see \cite{MV}. This is only sharp for $p=6$. The simplest argument to prove it uses  Bourgain's $L^6$ inequality
	\begin{equation}
		\label{strongestknown}
		\|\sum_{n=1}^Na_ne(nx+n^2t)\|_{L^6([0,1]^2,dxdt)}\les \|a_n\|_{l^2}
	\end{equation}
	and the fact that $|\sum_{n=1}^Na_ne(nx+n^2t)|$ is essentially constant in $t$ at scale $1/N^2$.

	No level set estimates \eqref{jcfoiepfcr[p,goo,g} exist in the literature, neither do they follow from the moment estimate \eqref{strongestknown} or from the known $N^{1/3}$ bound in Conjecture \ref{Maincon}. On the other hand, in the last section we show how this latter bound is easily recovered using our estimates.

	In Section \ref{sec2} we use a variant of the $TT^*$
	method to reduce Conjecture \ref{t1} to a (potentially stronger) conjecture with deep arithmetic flavor. We are left with a question that involves the quadratic  kernel $\K$. Using the known sharp estimates for its size, we are led to considering Conjecture \ref{c1}. The rest of the paper is devoted to verifying this conjecture in a certain range. We develop and operate in a rather combinatorial framework that takes advantage of sharp number theoretical input. In  Section \ref{sub9} we refine the combinatorics to suggest a further improvement to the range $\lambda\ges N^{\frac14+\frac1{12}}$, subject to some potentially difficult solution counting. Apart from this latter step of purely number theoretic nature, the rest of the argument works (and in fact is explicitly written) for $\lambda\ges N^{\frac14+\frac1{12}}$. Further improvements of this range seem possible with a more efficient use of various structures that arise in the argument. 
	
	The  paper \cite{FRW} shows that the upper bound $N^{1/3}$ is sharp for Conjecture \ref{Maincon}, within the general framework of convex sequences that fall under the scope of $l^2$ decoupling, \cite{BD}. Our methods here exploit the specific nature of the kernel $\K$, and are fundamentally different from decoupling. Notably, our argument involves no Fourier analysis.
	
	We would like to mention that Conjecture \ref{Maincon} has been verified when all $a_n$ are equal to 1, first by Barron \cite{Ba} with epsilon losses, and shortly after by Baker \cite{Bak}. The latter identified the precise magnitude
	$$\|\sup_{t}|\sum_{n=1}^Ne(nx+n^2t)|\|_{L^p([0,1],dx)}\sim N^{a(p)}(\log N)^{b(p)}$$
	for each $p\ge 1$.
	
	Specializing to other sequences $a_n$, Theorem \ref{teorem1} leads to some potentially interesting consequences. We mention one of them, for reader's amusement. It may hint about the number theoretic flavor of our argument, and about the potential difficulties involved in proving the full Conjecture \ref{t1}.
	\begin{co}
		Let $\Sc_N$ be a subset of $\{1,2,\ldots,N\}$ with density $\ges 1$, such as the primes, or any of their subsets of logarithmic density. Then for each $\theta\in[0,1]$
		$$|\{x\in[0,1]:\;\sup_{t}|\sum_{n\in\Sc_N}e(nx+n^2t+n^3\theta)|\ge \lambda\}|\les \frac{N}{\lambda^4}$$
		for each $\lambda\ges N^{\frac14+\frac1{10}}$.
	\end{co}
	\medskip
	
	\begin{ack*}
		I am grateful to Alex Barron for sharing with me a Fourier analytic proof of Conjecture \ref{c1} in the range $M\ggg N^{1/8}$. I have also benefited from conversations with Ruixiang Zhang regarding graph theoretic aspects of the problem. I would like to thank Alexandru Zaharescu for discussions related to Conjecture \ref{conditional}. I am grateful to the referee for the careful reading and for valuable suggestions. 	
	\end{ack*}

	\textbf{Notation.} Throughout this paper the main scale parameter will be denoted by $N$. The following notation will be used when comparing various other parameters $A=A_N,B=B_N$, as $N\to\infty$.
	
	We will write either $A=o(B)$ or $A\ll B$ (or $B\gg A$) to denote the fact that $A\le cB$ for some small enough absolute constant $c>0$ (independent of $N$).
	
	The notation  $A=O(B)$ or $A\lesssim B$ (or $B\gtrsim A$) will refer to the inequality $A\le CB$, for some potentially large, but still universal, constant $C$.
	
	We write $A\sim B$ if $A\lesssim B$ and $B\lesssim A$.
	
	We write $A\les B$ (or $B\ges A$), whenever $A\lesssim_\epsilon N^\epsilon B$ for all $\epsilon>0$.
	We write $A\approx B$ if $A\les B$ and $B\les A$.
	
	Finally, we write $A\lll B$ (or $B\ggg A$), whenever $A\lesssim N^{-\epsilon} B$ for some $\epsilon>0$. Note that $A\lll B$
	and $A\ges B$ are mutually exclusive.

	Throughout the argument we will use the fact that
	\begin{equation}
		\label{transitiviti}
		A\les B\lll C\implies A\lll C.\end{equation}

	\section{Initial reductions}
	\label{sec2}
	Let $w$ be an arbitrary Schwartz function supported on $[-2,2]$. It will suffice to prove
	$$|\{x:\;\sup_{t}|\sum_{n\in\Z}a_nw(\frac{n}{N})e(nx+n^2t)|\ge \lambda\}|\les \frac{N}{\lambda^4}.$$
	
	We write $\lambda=MN^{1/4}$, with $1\le M\le N^{\frac14}$. As a consequence of the Fourier Uncertainty Principle, the function $|\sum_{n\in\Z}a_nw(\frac{n}{N})e(nx+n^2t)|$ is essentially constant for $(x,t)$ in any box of size $1/N\times 1/N^2$.
	We fix a choice of $1/N$-separated $x_r\in [0,1]$ and pick $t_r\in [0,1]$, $1\le r\le R$ such that
	\begin{equation}
		\label{e1}
		|\sum_{n\in\Z}a_nw(\frac{n}{N})e(nx_r+n^2t_r)|\ge MN^{1/4}.
	\end{equation}
	We need to prove that
	$$R\les \frac{N}{M^4}.$$ 	
	\begin{re}
		\label{duiweyfucfi90gi}		
		It can be easily shown using $L^2$ orthogonality that each $t_r$  is associated with at most $\lesssim N^{1/2}/M^2$ points $x_r$.
	\end{re}
	
	\smallskip
	
	The rest of the argument will depend critically on the way we handle inequality \eqref{e1}.
	Let $|c_r|=1$ be such that
	$$|\sum_{n\in\Z}a_nw(\frac{n}{N})e(nx_r+n^2t_r)|=c_r\sum_{n\in\Z}a_nw(\frac{n}{N})e(nx_r+n^2t_r).$$
	Using Cauchy--Schwarz we find
	\begin{align*}
		RMN^{1/4}&\le \sum_{r=1}^Rc_r\sum_{n\in\Z}a_nw(\frac{n}{N})e(nx_r+n^2t_r)\\&=\sum_{n\in\Z}a_n\sum_{r=1}^Rc_rw(\frac{n}{N})e(nx_r+n^2t_r)\\&\le (\sum_{n\in\Z}|\sum_{r=1}^Rc_rw(\frac{n}{N})e(nx_r+n^2t_r)|^2)^{1/2}.
	\end{align*}
	Expanding the square and using the triangle inequality leads to
	$$RMN^{1/4}\le (\sum_{r,r'}|\sum_{n\in\Z}w(\frac{n}{N})^2e(n(x_r-x_{r'})+n^2(t_r-t_{r'}))|)^{1/2}.$$
	Writing $z_r=(x_r,t_r)$ and defining the weighted Gauss sum
	$$\K(x,t)=\sum_{n\in\Z}w(\frac{n}{N})^2e(nx+n^2t),$$
	we arrive at
	\begin{equation}
		\label{e2}
		R^2M^2N^{1/2}\le \sum_{r,r'}|\K(z_r-z_{r'})|.
	\end{equation}
	The derivation of \eqref{e2} is similar to the argument in \cite{Bo}. Note that this argument eliminates the coefficients $a_n$ and reduces things to a question about the kernel  $\K$. It is a close variant of the $TT^*$ method.
	In \cite{Demsing}, we used the $TT^*$ method  to prove sharp results for a class of special cases in Conjecture \ref{Maincon}. 
	\medskip

	To exploit \eqref{e2}, we recall the following classical upper bounds for $\K$, see e.g. Lemma 9.2 in \cite{De1}.  Due to periodicity, we may replace the range $[0,1]$ for $t$ with $[-\frac1N,1-\frac1N]$. Letting $J_{1,0}=[-\frac1N,\frac1N]$ and $J_{q,a}=[\frac{a}q-\frac1{Nq},\frac{a}q+\frac1{Nq}]$ for each $2\le q\le N$ and $1\le a\le q-1$ with $(a,q)=1$, we recall that these intervals cover $[-\frac1N,1-\frac1N]$. The intervals may overlap, but that will be irrelevant for our argument.
	\begin{lem}
		\label{22}
		
		Assume $t\in J_{q,a}$, and assume that either $|t-\frac{a}q|\sim \frac{1}{2^lNQ}$ for $2^{l}<N/Q$ or that $|t-\frac{a}q|\lesssim \frac{1}{2^lNQ}=\frac{1}{N^2}$ when $2^l=N/Q$. Then for each $\epsilon>0$ we have
		$$|\K(x,t)|\lesssim {N^{1/2}}2^{l/2}$$
		if  $$x\in \Mc=\bigcup_{b\in\Z}\;[\frac{b}{q}- \frac{N^{\epsilon}}{Q2^l},\frac{b}{q}+ \frac{N^{\epsilon}}{Q2^l}],$$
		and
		$$|\K(x,t)|\lesssim_\epsilon N^{-100}$$
		if  $$x\not\in \Mc.$$
		
	\end{lem}
	Let us write $\Ac(1)=\{0\}$, and for $2\le q\le N$
	$$\Ac(q)=\{1\le a\le q-1:\;\gcd(a,q)=1\}.$$
	Fix $\epsilon$. All further uses of $\les$ will hide implicit dependence on this particular $\epsilon$. For dyadic $1\le Q\le N$ and for $1\le 2^l\lesssim \frac{N}{Q}  $ we write
	$$S_{Q,l}=\bigcup_{q\sim Q}\bigcup_{a\in \Ac(q)}\bigcup_{0\le b\le q-1}\;[\frac{b}{q}-\frac{N^\epsilon}{2^lq},\frac{b}{q}+\frac{N^\epsilon}{2^lq}]\times [\frac{a}q-\frac{1}{2^lqN},\frac{a}q+\frac{1}{2^lqN}].$$
	The previous lemma shows that for each $(x,t)\in \R^2$ we have
	$$|\K(x,t)|\lesssim 2^{l/2}N^{1/2}\sum_{1\le Q\le N}\sum_{1\le 2^l\lesssim N/Q}\tilde{1}_{S_{Q,l}}(x,t)+ N^{-100}.$$
	We have denoted by $\tilde{1}_{S_{Q,l}}$ the $1\times 1$-periodic extension of $1_{S_{Q,l}}$.
	
	\begin{re}
		\label{dyad}
		In fact, slightly more is true, and this will be used later. Writing
		$$S_{Q,l}^{dyad}=\bigcup_{q\sim Q}\bigcup_{a\in \Ac(q)}\bigcup_{0\le b\le q-1}\;[\frac{b}{q}-\frac{N^\epsilon}{2^lq},\frac{b}{q}+\frac{N^\epsilon}{2^lq}]\times \{t:\;|t-\frac{a}q|\sim \frac{1}{2^lqN}\}$$
		for $2^l<N/Q$ and
		$$S_{Q,l}^{dyad}=S_{Q,l}$$
		for $2^l=N/Q$,
		we still have
		$$|\K(x,t)|\les 2^{l/2}N^{1/2}\sum_{1\le Q\le N}\sum_{1\le 2^l\lesssim N/Q}\tilde{1}_{S^{dyad}_{Q,l}}(x,t)+ N^{-100}.$$
	\end{re}

	Combining this with \eqref{e2} and pigeonholing, we find that there must be a $Q$ and some $l$ such that
	$$R^2\les \frac{2^{l/2}}{M^2}\sum_{r,r'}\tilde{1}_{S_{Q,l}}(z_r-z_r').$$
	We note that we must have
	\begin{equation}
		\label{e5}
		M^4\les 2^l,\;\text{and thus also }Q\les \frac{N}{M^4}.
	\end{equation}
	At this point, we make the following conjecture.
	\begin{con}\label{c1}Assume $\|a_n\|_{l^2}=1$.
		Assume $1\le M\le N^{1/4}$ and $1\le 2^l\lesssim \frac{N}{Q}$.
		
		Let $z_r=(x_r,t_r)$, $1\le r\le R$, with $x_r$ $1/N$-separated, satisfy 	
		\begin{equation}
			\label{e1again}
			|\sum_{n\in\Z}a_nw(\frac{n}{N})e(nx_r+n^2t_r)|\ge MN^{1/4}\end{equation}
		and
		\begin{equation}
			\label{e7}
			R^2\les \frac{2^{l/2}}{M^2}\sum_{r,r'}\tilde{1}_{S_{Q,l}}(z_r-z_r').\end{equation}
		Then we have
		\begin{equation}
			\label{e7600}
			R\les M^{-4}N.
		\end{equation}
	\end{con}
	
	The previous discussion can be summarized as follows.
	\begin{pr}
		If Conjecture \ref{c1} holds for some $M$ (and all $l,Q$), then the level set estimate  $$|\{x\in [0,1]:\;\sup_{t}|\sum_{n=1}^Na_ne(nx+n^2t)|\ge \lambda\}|\les \frac{N}{\lambda^4}$$
		also holds for $\lambda\sim MN^{1/4}$.
	\end{pr}
	Conjecture \ref{c1} is true when $M\ggg N^{1/8}$, without hypothesis \eqref{e1again}, and with \eqref{e7} replaced by the weaker hypothesis
	\begin{equation}
		\label{e7weaker}
		R^2\les \frac{2^{l/2}}{M^2}\sum_{r,r'}\tilde{1}_{X_{Q,l}}(x_r-x_r'),\end{equation}
	where
	$$X_{Q,l}=\bigcup_{q\sim Q}\bigcup_{0\le b\le q-1}\;[\frac{b}{q}-\frac{N^\epsilon}{2^lq},\frac{b}{q}+\frac{N^\epsilon}{2^lq}].$$
	The proof (cf. \cite{Ba2}) follows the Fourier argument in \cite{Bo} and does not make use of the $t$ component.
	
	If only the weaker hypothesis \eqref{e7weaker} for the $x$ component is assumed, \eqref{e7600} is false when $M\les N^{1/8}$.  Indeed, assume $Q=2^l=N^{1/2}/100$. Then it is easy to see that $X_{Q,l}$ has measure $\sim 1$. Indeed, there are  $\sim Q^2\sim N$ distinct fractions $b/q$ with $q\sim Q$, and the corresponding intervals $[\frac{b}{q}-\frac{1}{N}, \frac{b}{q}+\frac{1}{N}]$ are pairwise disjoint subsets of $X_{Q,l}$.
	Using this and the fact that $X_{Q,l}$ is a union of intervals of length  $\ge 1/N$, we guarantee that $X_{Q,l}$ contains  $\sim N$ of the points $j/N$, $j\in\{1,\ldots, N\}$. For each such $j$, there are exactly $N$ pairs $(r/N,r'/N)$ with $r,r'\in\{1,\ldots, N\}$, such that $\frac{r}{N}-\frac{r'}{N}\equiv \frac{j}{N}\pmod 1$. This implies that $\tilde{1}_{X_{Q,l}}(\frac{r}{N}-\frac{r'}{N})=1$ for $\sim N^2$ such pairs. Thus, taking $x_r=r/N$ for $1\le r\le N$,  \eqref{e7weaker} is trivially satisfied since $M^4\les 2^l$. On the other hand, $R\les M^{-4}N$ fails when $M\ggg 1$, since $R=N$.
	\medskip

	Let us discuss a sharp example for Conjecture \ref{c1}. Fix a prime $q\sim Q$. Consider a maximal family of $1/N$-separated points $x_i$, each lying in the $1/Q2^l$ neighborhood of some $b/q$, $0\le b\le q-1$. There are $\sim N/2^l$ such  $x_i$. Consider the points $z_{i,j}=(x_{i,j},t_{i,j})=(x_i+\frac{j}{Q2^l},b/q)$, with $b$ the numerator corresponding to $x_i$ and $2\le j\le J$, for some $J\ll 2^l$.
	It follows that $x_{i,j}$ are $1/N$-separated. There are $R\sim JN/2^l$  points $z_{i,j}$.
	Note that if $x_i$ is near $b/q$ and $x_{i'}$ is near $b'/q$ with $b\not=b'$ then $\tilde{1}_{S_{Q,l}}(z_{i,j}-z_{i',j})=1$ for each $j$. There are $\sim \frac{1}{J}R^2$ such pairs.
	Thus, \eqref{e7} holds if $J\sim 2^{l/2}/M^2$. Note also that \eqref{e7600} is satisfied, in fact we have the superficially stronger estimate $R\sim \frac{N}{2^{l/2}M^2}$.
	\smallskip
	
	For the rest of the paper, we will drop the term $N^\epsilon$ in the definition of $S_{Q,l}$. We do so only to simplify notation, this has no impact on the argument.
	
	It remains to prove Conjecture \ref{c1} for $M\ges N^{1/10}$. We find it plausible that the conjecture should be true in the full range of $M$, in fact even without the hypothesis \eqref{e1again}. The reason we include \eqref{e1again} is because our argument in Section \ref{sec:four}  makes use of it at some point.
	
	\smallskip

	\section{A combinatorial approach for Conjecture \ref{c1} in the range $M\ges  N^{1/10}$}
	\label{sec:four}
	\bigskip
	
	This section has a few subsections, all contributing to a complete proof of Conjecture \ref{c1}  in  the range $M\ges N^{1/10}$.
	\bigskip

	We will use the notation for the density parameter 
	\begin{equation}
	\label{justK}
	K=\frac{2^{l/2}}{M^2}.\end{equation} 
	Recall our setup \eqref{e7}
	$$\frac{R^2}{K}\les |\{(r,r'):\;z_r-z_{r'}\in S_{Q,l}\}|,$$
	which forces 
	$$K\ges 1.$$
	
	We need to prove that $R\les N/M^4$. We may assume $K\ll \sqrt{R}$, or else $R\lesssim K^2=2^l/M^4\le N/M^4$, and we are done.
	
	\subsection{A brief outline of the argument}
	
	Here is an outline of our approach. We construct a graph $G$ with vertex set $(z_r:\;1\le r\le R)$,  by drawing an edge between $z_r$ and $z_{r'}$ if $z_r-z_{r'}\in S_{Q,l}.$
	By deleting edges if necessary, we may assume that
	\begin{equation}
		\label{djjfvji}
		E(G)\approx R^2/K.
	\end{equation}
	We focus attention on the \textit{popular} pairs $r,r'$, those that share at least $R/K^2$  neighbors $r''$.
	The equality
	\begin{equation}
		\label{keysum}
		z_r-z_{r'}=(z_r-z_{r''})+(z_{r''}-z_{r'})
	\end{equation}
	guarantees that $z_{r}-z_{r'}$ has many representations as sums of elements in $S_{Q,l}$. Each of these is associated with a unique triple $q,a,b$, say
	$$z_r-z_{r''}\in [\frac{b_1}{q_1}-\frac{1}{2^lQ},\frac{b_1}{q_1}+\frac{1}{2^lQ}]\times [\frac{a_1}{q_1}-\frac{1}{2^lQN},\frac{a_1}{q_1}+\frac{1}{2^lQN}]$$
	$$z_{r''}-z_{r'}\in [\frac{b_2}{q_2}-\frac{1}{2^lQ},\frac{b_2}{q_2}+\frac{1}{2^lQ}]\times [\frac{a_2}{q_2}-\frac{1}{2^lQN},\frac{a_2}{q_2}+\frac{1}{2^lQN}].$$
	We will indicate this by writing
	\begin{equation}
		\label{ucufijiodadkod[pkc]}
		(z_r-z_{r''})\blacksquare (a_1,b_1,q_1),\;\;(z_{r''}-z_{r'})\blacksquare(a_2,b_2,q_2).\end{equation}
	
	The key task is this: given $z_r,z_{r'}$, find a good upper bound on the number of such $z_{r''}$. This has the following reformulation, upon observing that the value of $b_1/q_1$ (or alternatively $b_2/q_2$) determines the value of $z_{r''}$ up to $N/Q2^l$ choices (due to the $1/N$-separation of the $x_{r''}$, and the fact that $x_{r''}$ determines $z_{r''}$).
	
	\begin{de}
		\label{enumbers}	
		We cover $[0,1]^2$ with pairwise disjoint boxes $B=I\times J\subset [0,1]^2$, with $|I|\sim 1/2^lQ$ and $|J|\sim 1/N2^lQ$. For each $B$, let  $n_B$ be the number of pairs $(b_1,q_1)$ with $|b_1|\le q_1-1\sim Q$ such that there are integers $a_1,a_2,b_2,q_2$ satisfying $|b_2|,|a_2|\le q_2-1\sim Q$, $|a_1|\le q_1-1$, $\gcd(a_1,q_1)=\gcd(a_2,q_2)=1$ and
		$$(\frac{b_1}{q_1}+\frac{b_2}{q_2},\frac{a_1}{q_1}+\frac{a_2}{q_2})\in B.$$
		If in this definition we also require that $\gcd(q_1,q_2)=1$, the number of pairs will be denoted by $n_B^*$. Note that $n_B^*\le n_B$.
		
	\end{de}
	
	We can easily close the argument if we find a popular pair $z_r,z_{r'}$, such that if $B$ is the box for which $z_r-z_{r'}\in B$, we have
	$$n_B\les Q.$$
	Indeed, as observed earlier, this forces \eqref{keysum} to have at most $\les Q\times \frac{N}{Q2^l}=\frac{N}{2^l}$ solutions $z_{r''}$. On the other hand, the pair being popular shows that there are at least $R/K^2$ solutions. Putting the two things together
	\begin{equation}
		\label{owoeieicieieci}
		\frac{R}{K^2}\les \frac{N}{2^l}\end{equation}
	leads to the desired upper bound $R\les N/M^4$.
	
	As we shall soon see, the intended bound $n_B\les Q$ does not hold for all boxes $B$, so the argument outlined in the previous paragraph will not suffice. But it will be crucial for the upper bound to hold in average, for the ``generic" box. The next result  shows that this is the case precisely when $M\ges N^{1/10}$. This justifies why our argument here can only address this regime.
	
	\begin{pr}
		\label{averagevalues}
		The sum over all ($\sim NQ^22^{2l})$ boxes satisfies
		$$\sum_{B} n_B\lesssim Q^6.$$
		Thus, when $M\ges N^{1/10}$,  the average value of $n_B$ satisfies the intended upper bound
		$$\frac{1}{\sharp B}\sum_Bn_B\les Q$$
		for all admissible $Q, l$ (satisfying \eqref{e5}).
		
		If $M\ggg N^{1/12}$, we have the lower bound
		$$\sum_{B} n_B^{*}\ges Q^6.$$
		Thus, when $N^{1/10}\ggg M\ggg N^{1/12}$ the average value of $n_B^*$ (and thus also that of $n_B$) satisfies the lower bound
		$$\frac{1}{\sharp B}\sum_Bn_B^{*}\ggg Q$$
		for some admissible $Q, l$ (satisfying \eqref{e5}).
	\end{pr}
	\begin{proof}
		Write
		$$N_B=|\{(a_1,a_2,b_1,b_2,q_1,q_2):\;q_i\sim Q,\;\gcd(a_i,q_i)=1,\;(\frac{b_1}{q_1}+\frac{b_2}{q_2},\frac{a_1}{q_1}+\frac{a_2}{q_2})\in B \}|$$
		$$N_B^*=|\{(a_1,a_2,b_1,b_2,q_1,q_2):\;q_i\sim Q,\;\gcd(a_i,q_i)=1, $$
		$$(\frac{b_1}{q_1}+\frac{b_2}{q_2},\frac{a_1}{q_1}+\frac{a_2}{q_2})\in B\text{ and }\gcd(q_1,q_2)=1\}|.$$
		It is immediate that  $\sum_B N_B\sim Q^6$ and that $n_B\le N_B,$ leading to
		$$\sum_{B} n_B\lesssim Q^6.$$
		Our assumption  $2^l\ges M^4\ges N^{2/5}$ implies that $Q^3\lesssim N^3/2^{3l}\les N2^{2l}$, and thus $Q^6/NQ^22^{2l}\les Q$.

		We next prove that $n_B^*\gtrsim N_{B}^*$ when $M\ggg N^{1/12}$. We will in fact prove that if
		$$(\frac{b_1}{q_1}+\frac{b_2}{q_2},\frac{a_1}{q_1}+\frac{a_2}{q_2}),\;(\frac{b_1}{q_1}+\frac{b_2'}{q_2'},\frac{a_1'}{q_1}+\frac{a_2'}{q_2'})\in B$$
		with $\gcd(q_1,q_2)=\gcd(q_1,q_2')=1$,
		then $(a_1,a_2,b_2,q_2)=(a_1'\pm q_1,a_2'\pm q_2,b_2',q_2')$. In other words, given $B$, the pair $(b_1,q_1)$ determines the remaining entries $a_1,a_2,b_2,q_2$ up to $O(1)$ many choices.
		
		We first use that $Q^3\ll NQ2^l$ when $M\ggg N^{1/12}$. Thus,
		$$\frac{a_1}{q_1}+\frac{a_2}{q_2}=\frac{a_1'}{q_1}+\frac{a_2'}{q_2'}+O(1/NQ2^l)$$
		forces
		$$\frac{a_1}{q_1}+\frac{a_2}{q_2}=\frac{a_1'}{q_1}+\frac{a_2'}{q_2'},$$
		or
		$$(a_1'-a_1)q_2q_2'=q_1(a_2q_2'-a_2'q_2).$$
		Thus $q_2|q_1q_2'$ and $q_2'|q_1q_2$, and since $\gcd(q_1,q_2)=\gcd(q_1,q_2')=1$, it follows that $q_2=q_2'$.
		
		Second, since $\frac{b_2}{q_2}-\frac{b_2'}{q_2'}=O(\frac1{Q2^l})$, it follows that $b_2=b_2'$.
		
		Third,  $\frac{a_1}{q_1}+\frac{a_2}{q_2}=\frac{a_1'}{q_1}+\frac{a_2'}{q_2}$ together with $|a_1-a_1'|< 2q_1,|a_2-a_2'|<2q_2$ and $\gcd(q_1,q_2)=1$ forces $a_1=a_1'\pm q_1$, $a_2=a_2'\pm q_2$.
		
		It is easy to see that (for example, assume $q_1,q_2$ are both primes)
		$$\sum_B N_B^{*}\ges Q^6,$$
		which concludes the lower bound
		$$\sum_B n_B^*\ges Q^6.$$

		When $M\lll N^{1/10}$ we may find admissible $l,Q$  (satisfying \eqref{e5}) such that $Q^6/NQ^22^{2l}\ggg Q$.

	\end{proof}
	
	This proposition shows that the average value of $n_B^*$ is $\approx Q^4/N2^{2l}$. In the next subsection we will see that for each $B$ we have $n_B^*\les Q^4/N2^{2l}$, see \eqref{sharppppp}.
	\smallskip

	Let us now prove the existence of ``enemies", or boxes $B$ for which $n_B$ is much larger than $Q$. In fact we show more. We may define $\tilde{n_B}$ by requiring an extra condition in the definition of $n_B$. Namely, we may insist that the fractions $b_1/q_1$ generated by the pairs $(b_1,q_1)$ counted by $n_B$ are separated by at least $1/N$. Obviously $\tilde{n_B}\le n_B$. The earlier argument would still work if we had the bound $\tilde{n_B}\les Q$ for the box containing the difference $z_r-z_{r'}$ of one particular popular pair. The next example shows that even $\tilde{n_B}$ can sometimes be much larger than $Q$.
	
	\begin{ex}[Enemies]
		\label{exenemies}
		Fix $1\le q\le Q$, $1\le r\le Q/q$ with $rq\ge Q^2/N$.  Let $t=\frac{a}{q}$ with $0\le a\le q-1$ (no requirement of the form $(a,q)=1$ is needed). Let also $x=\frac{b}{rq}$, the value of $0\le b\le q-1$ is  arbitrary.
		
		We will show that there are solutions to
		$$t=\frac{a_1}{q_1}+\frac{a_2}{q_2}$$
		$$x=\frac{b_1}{q_1}+\frac{b_2}{q_2},$$
		with $q_1,q_2\sim Q$, $\gcd(a_1,q_1)=\gcd(a_2,q_2)=1$,
		with the points $b_1/q_1$ being $1/N$-separated and with at least $\ges Q^2/rq$ many such pairs $(b_1,q_1)$. This shows that if $B$ is the box in which the pair $(x,t)$ lies, then $\tilde{n_B}\ges Q^2/rq$. This lower bound can be enforced to be much larger than $Q$.
		\smallskip

		Indeed, note that for each $n=rm\sim Q/q$ we may take $q_1=q_2=nq\sim Q$ and pick $a_1$, $a_2=na-a_1$ relatively prime to $nq$ and such that
		$$\frac{a}q=\frac{a_1}{nq}+\frac{na-a_1}{nq}.$$
		Then note that
		$$\frac{b}{rq}=\frac{b_1}{nq}+\frac{mb-b_1}{nq}$$
		holds for all $b_1\le nq-1$. It remains to see that the distinct fractions
		$$\frac{b_1}{nq}:\;\gcd(b_1,nq)=1,\;n=rm\sim Q/q$$
		are at least as many as $\ges Q^2/rq$,   and are $1/N$-separated. The first claim follows since Euler's totient function satisfies $\phi(m)\ges m$. For the second claim, note that if $\frac{b_1}{nq}$ and $\frac{b_1'}{n'q}$ are distinct, their difference is at least
		$$\frac{1}{mm'rq}\sim \frac{qr}{Q^2}.$$
		This is greater than $1/N$ by hypothesis. See Remark \ref{enemiesexplained} for the sharpness of this estimate.

	\end{ex}	
	
	In Section \ref{s5} we introduce parameters $D,P,F$ that will allow us to find sensible upper bounds for (a refined version of) $n_B$ in Section \ref{s6}. In a certain range of these parameters we will have the favorable bound $n_B\les Q$, and the earlier argument based on \eqref{owoeieicieieci} will work. In the remaining range of parameters where such a bound is not available, we will use pigeonholing to isolate a certain piece in the graph, that we call {\em fork}. This fork is proved to have additional arithmetic structure, that will be amenable to sharp  counting arguments.

	\subsection{Pigeonholing and graph considerations}
	\label{s5}
	
	The rest of the analysis is driven by the size of the greatest common divisor of $q_1,q_2$ denoted by  $\gcd(q_1,q_2)$.
	Let $D$ be a dyadic number $1\le D\le Q$.
	{Pigeonholing}, we may assume that a big chunk (by that we mean the total number  divided by $\log N$) of the popular pairs have a big chunk of their joint neighbors $z_{r''}$ satisfy \eqref{ucufijiodadkod[pkc]} with their greatest common divisor satisfying $\gcd(q_1,q_2)\sim D$. There are two other parameters $P,F$ that will also be pigeonholed.

	\bigskip
	
	\begin{lem}
		\label{primenessofp}	
		Assume $\gcd(q_1,q_2)=d$, so $q_1=dm_1$, $q_2=dm_2$, $\gcd(m_1,m_2)=1$. Assume also that $\gcd(a_1,q_1)=\gcd(a_2,q_2)=1$. Let $$p=\gcd(\frac{a_1q_2+a_2q_1}{d},\frac{q_1q_2}{d}).$$
		Then $p|d$, so in particular $p\le d$. Also, $\gcd(p,m_1)=\gcd(p,m_2)=1$.
	\end{lem}
	\begin{proof}
		Any prime factor $p_1$ of $p$ must divide $dm_1m_2$ and $m_2a_1+m_1a_2$. Assume for contradiction that it divides $m_1$. Then it must also divide $m_2a_1$. But since $(a_1,m_1)=1$, we are left with $p_1|m_2$, contradicting that $(m_1,m_2)=1$.
		
	\end{proof}
	
	\bigskip
	
	The relevance of $p$ is that it determines the denominator $q$ in the least terms representation
	$$\frac{a_1}{q_1}+\frac{a_2}{q_2}=\frac{A}{q},\;\;q=\frac{d}{p}m_1m_2,\;\gcd(A,q)=1.$$
	
	Given $b_1,b_2$ with $|b_1|\le q_1-1$, $|b_2|\le q_2-1$, we will also consider the parameter $f$ given by
	\begin{equation}
		\label{iojurifuriurugitug}
		f=\gcd(\frac{b_1q_2+b_2q_1}{d},p).\end{equation}
	Note that we may write
	$$\frac{b_1}{q_1}+\frac{b_2}{q_2}=\frac{B}{\frac{p}{f}q},\;\;q=\frac{d}{p}m_1m_2.$$
	While $\gcd(B,p/f)=1$, $B$ and $q$ are allowed to share factors. $B$ might be equal to zero, in which case $f=p$. The parameter $f$ will play a critical role in the counting arguments in Subsections \ref{sub3} and \ref{sub7}. See Remark \ref{roleoffexplained} for the subtlety behind our choice of $f$. See also Remark \ref{roleoff} for the key difference between $p$ and $f$.
	
	\bigskip
	
	Let us get to the details of the argument. Recall the graph $G$ introduced earlier. We will call its vertices $z_r$ with the letter $v$.
	
	For two vertices $v_1\not=v_2$ in $G$, write $\Nc(v_1,v_2)$ for the common neighbors. Define the popular pairs
	$$\Pc=\{(v_1,v_2):\;|\Nc(v_1,v_2)|\gtrsim R/K^2\}.$$

	We rely on the following easy result.
	\begin{lem}
		For each graph $G$ with $R$ vertices $V(G)$  and   $R^2/K$ edges $E(G)$, $K\ll \sqrt{R}$,  we have
		\begin{equation}
			\label{forgitgii609i}
			\sum_{(v_1,v_2)\in\Pc}|\Nc(v_1,v_2)|\gtrsim R^3/K^2.
		\end{equation}
	\end{lem}
	\begin{proof}
		Writing $\Nc_v$ for the neighbors of $v$, we note that	
		$$\|\sum_{v\in V(G)} 1_{\Nc_v}\|_1=2E(G)\ge R^2/K.$$
		Using H\"older we get
		$$\|\sum_{v\in V(G)} 1_{\Nc_v}\|_2^2\ge \frac1R\|\sum_{v\in V(G)} 1_{\Nc_v}\|_1^2\ge \frac{R^3}{K^2}.$$
		On the other hand
		$$\|\sum_{v\in V(G)} 1_{\Nc_v}\|_2^2=2E(G)+2\sum_{v_1\not=v_2}|\Nc(v_1,v_2)|<\frac{R^3}{2K^2}+2\sum_{v_1\not=v_2}|\Nc(v_1,v_2)|.$$
		We conclude that
		\begin{equation}
			\label{e15'}
			\sum_{v_1\not=v_2}|\Nc(v_1,v_2)|>\frac{R^3}{2K^2}.
		\end{equation}
		Let
		$$\Vc=\{(v_1,v_2):\;v_1\not=v_2,\;|\Nc(v_1,v_2)|< \frac{R}{4K^2}\}.$$
		Since
		\begin{equation}
			\label{e14}
			\sum_{(v_1,v_2)\in\Vc}|\Nc(v_1,v_2)|<\frac{R}{4K^2}R^2,\end{equation}
		the conclusion follows by combining \eqref{e15'} and \eqref{e14}.
		
	\end{proof}

	\bigskip
	
	For dyadic $1\le F\le P\le D\le Q$, let
	$$\Nc_{D,P,F}(v_1,v_2)=\{v_3\in\Nc(v_1,v_2):\;$$$$(v_1-v_3)\blacksquare (a_1,b_1,q_1),\;(v_3-v_2)\blacksquare (a_2,b_2,q_2),\text{ such that }$$$$\gcd(q_1,q_2):=d\sim D\text{ and }
	$$$$
	\gcd(\frac{a_1q_2+a_2q_1}{d},\frac{q_1q_2}{d}):=p\sim P\text{ and}$$$$\gcd(\frac{b_1q_2+b_2q_1}{d},p):=f\sim F\}.$$
	Note the partition
	$$\Nc(v_1,v_2)=\bigcup_{D,P,F}\Nc_{D,P,F}(v_1,v_2).$$
	Call
	$$\Pc_{D,P,F}=\{(v_1,v_2)\in\Pc:$$$$\;|\Nc_{D,P,F}(v_1,v_2)|=\max_{D',P',F'}|\Nc_{D',P',F'}(v_1,v_2)|,\;\text{and }(D,P,F)\text{ is  smallest lexicographically}\}.$$
	Note that the sets $\Pc_{D,P,F}$ partition $\Pc$.
	Note that if $(v_1,v_2)\in \Pc_{D,P,F}$ then
	$$|\Nc_{D,P,F}(v_1,v_2)|\ges |\Nc(v_1,v_2)|.$$
	It follows that
	\begin{align*}
		\sum_{D,P,F}\sum_{(v_1,v_2)\in\Pc_{D,P,F}}\Nc_{D,P,F}(v_1,v_2)&\ges \sum_{D,P,F}\sum_{(v_1,v_2)\in\Pc_{D,P,F}}|\Nc(v_1,v_2)|\\&=\sum_{(v_1,v_2)\in\Pc}|\Nc(v_1,v_2)|\ges R^3/K^2.
	\end{align*}
	Pigeonhole  $D,P,F$ such that
	\begin{equation}
		\label{fiu3u9fiewld;.}
		\sum_{(v_1,v_2)\in\Pc_{D,P,F}}\Nc_{D,P,F}(v_1,v_2)\ges R^3/K^2.
	\end{equation}
	Let us discuss three  examples with $K\sim 1$.
	\begin{ex}\label{exemplu1}
		In this example, the vertices $(x_r,t_r)$ of $G$ are taken to be $(a/q,a/q)$ with $|a|\le q-1$, for some fixed prime $q\sim Q$. We draw an edge between $(a_1/q,a_1/q)$ and $(a_2/q,a_2/q)$ if $\gcd(a_2-a_1,q)=1$. There are $R\sim Q$ vertices and $\sim  R^2$ edges. The parameters $D,P,F$ satisfying \eqref{fiu3u9fiewld;.} with $K\sim 1$ are  equal to $Q$, $1$ and $1$, respectively. Note that $R\lesssim N/M^4$. If $Q\sim N/2^l$ and $2^l\sim M^4$, we do in fact get a tight example where $R\sim N/M^4$.
	\end{ex}
	\smallskip
	
	\begin{ex}
		\label{interexam}	
		In this example, the vertices $(x_r,t_r)$ of $G$ are taken to be $(c/r,c/r)$ with $r$ ranging over all primes $\sim\sqrt{Q}$ and  $1\le |c|\le r-1$. We draw an edge between $(c_1/r_1,c_1/r_1)$ and $(c_2/r_2,c_2/r_2)$ if $r_1\not=r_2$. There are $R\approx Q$ vertices and $\approx R^2$ edges. The parameters $D,P,F$ satisfying \eqref{fiu3u9fiewld;.} with $K=1$ are all equal to $\sqrt{Q}$. See Example \ref{exsquareroot}. If $Q\sim N/2^l$ and $2^l\sim M^4$, we  get again a tight example,  $R\approx N/M^4$.
		\smallskip
		
		For each $\sqrt{Q}\ll D\le Q$, this example may be modified to produce a graph represented by this $D$. Indeed, fix a prime $1\ll d\le Q$, and consider vertices of the form $(\frac{b}{dq},\frac{1}{dq})$, with $q$ ranging over all primes $\sim \sqrt{Q/d}$, and with $\gcd(b,dq)=1$. This latter condition is meant to ensure the $1/N$-separation of the $x$-entries $b/dq$. Draw an edge between $(\frac{b}{dq},\frac{1}{dq})$ and $(\frac{b'}{dq'},\frac{1}{dq'})$ if $\gcd(q-q',d)=1$.
		There are $R\approx Q/D$ vertices and $\sim R^2$ edges. It can be easily checked that \eqref{fiu3u9fiewld;.} is satisfied with $K\sim 1$, $D\sim \sqrt{Qd}$ and $P=F\sim \sqrt{Q/d}$. However, this example is not tight.
		
	\end{ex}
	The following variant of the previous example leads to a construction that satisfies \eqref{fiu3u9fiewld;.} with two different choices for the pair $(P,D)$.
	\begin{ex}
		\label{exemplu3}	
		Fix dyadic numbers $ 1\le Q_1,Q_2\le Q$ with $Q_1Q_2=Q$ and $Q_1\not= Q_2$. For $i\in\{1,2\}$ let 	
		$V_i$ consist of points $v_i$ of the form $(\frac{b_i}{r_i},\frac{c_i}{r_i})$ with $r_i\sim Q_i$ prime, and  $1\le |b_i|, |c_i|\le r_i-1$.  Thus, $1\le |V_i|\lesssim Q_i^3$.
		
		Note that if $v_i\in V_i$ then $v_1-v_2$ is of the form $(\frac{b}{r_1r_2},\frac{a}{r_1r_2})$ with $\gcd(a,r_1r_2)=1$. Since $r_1r_2\sim Q$, we may draw an edge between each $v_1$ and each $v_2$. The bipartite graph has $R=|V_1|+|V_2|$ vertices and $|V_1||V_2|$ edges. Choosing $|V_1|\sim |V_2|$, we may ensure maximum edge density, $K=1$.
		
		Looking at this graph in two ways, \eqref{fiu3u9fiewld;.} is satisfied both with $P=D=Q_1$ and $P=D=Q_2$.  	
	\end{ex}
	
	\begin{re}
		\label{random}	
		We do not have an example of a  nontrivial  graph (that is, having $\ggg 1$ vertices) satisfying \eqref{fiu3u9fiewld;.} with $D\les 1$ and  $K\sim 1$. We find it likely that such a graph does not exist.
		To wit, in Example \ref{exemplu3} the requirement $D\les 1$ forces $Q_2\les 1$, if, say,  $Q_2\le Q_1$. Thus $|V_2|\les 1$. In order to keep the edge density $\approx 1$, this forces $|V_1|\les 1$, leading to a trivial graph.
	\end{re}
	It may be helpful to contrast these example with the random graph.
	\smallskip
	
	\begin{ex}[The random graph]
		  We select $R$ points $t_1,\ldots,t_R $ in $[0,1]$ independently with uniform probability. The vertex set of $G$ will consist of the points $v_r=(t_r,t_r)$, $1\le r\le R$.
		 
		 We draw an edge between $v_i$ and $v_j$ if $|t_i-t_j-a/q|\lesssim 1/NQ2^l$ for some $q\sim Q$ and $\gcd(a,q)=1$. The expected number of edges is $\sim R^2\frac{Q}{N2^l}$. Thus, the generic edge density of  this graph is $\sim \frac{Q}{N2^l}$. Note however that $\frac{Q}{N2^l}\lll 2^{-l/2}\le \frac{M^2}{2^{l/2}}$, unless $Q\approx N$. 
		  This shows that the random graph is far from achieving \eqref{djjfvji} in the nontrivial case when $2^l\ggg 1$. In other words, the random graph is a non-example. 
		\end{ex}
		  The graphs in Examples \ref{exemplu1}-\ref{exemplu3} perform   much better than the random graph by using arithmetic structures. There appears to be no analogue for these examples when $D\les 1$. Proving this would probably lead to further progress on Conjecture \ref{c1}.

	The parameters $D,P,F$ will be fixed throughout the rest of the argument.

	\subsection{Counting solutions }
	\label{s6}
	
	Recall that $1\le F\le P\le  D\le Q$.

	There are no restrictions on the size of the parameter $M\ge 1$ in this subsection. We will make the very mild assumption that $Q\ll N$, mostly for convenience.  
	
	The ultimate goal here is to prove Proposition \ref{p134}. This will be achieved in two stages, Lemma \ref{l7} and  Lemma \ref{l8}.
	
	\begin{de}
		Let $x,t$ be  arbitrary numbers.
		
		We will call a pair $(q_1,q_2)$ with $\gcd(q_1,q_2):=d\sim D$ admissible (for $x,t$, with respect to parameters $D,P,F$) if $q_1,q_2\sim Q$ and there are $a_j$ with  $|a_j|<q_j$ and $\gcd(a_j,q_j)=1$, and there are $b_j$ with $|b_j|<q_j$ for each $j\in\{1,2\}$ such that
		\begin{equation}
			\label{e15}
			|\frac{a_1}{q_1}+\frac{a_2}{q_2}-t|\lesssim  \frac1{2^lNQ},
		\end{equation}
		\begin{equation}
			\label{e15despreb}
			|\frac{b_1}{q_1}+\frac{b_2}{q_2}-x|\lesssim  \frac1{2^lQ},
		\end{equation}
		and
		\begin{equation}
			\label{e15desprep}
			p:=\gcd(\frac{a_1q_2+a_2q_1}{d},\frac{q_1q_2}{d})\sim P,
		\end{equation}
		\begin{equation}
			\label{e15despref}
			f:=\gcd(\frac{b_1q_2+b_2q_1}{d},p)\sim F.
		\end{equation}
	\end{de}

	\begin{re}
		\label{invariant}	
		An admissible pair $(q_1,q_2)$ may satisfy \eqref{e15}-\eqref{e15despref} with different values of $a_1,a_2,b_1,b_2$. The pair $(q_1,q_2)$ uniquely determines  the parameter $p$ (in addition to $d$), via \eqref{e15}. Indeed, assume
		$$|\frac{a_1'}{q_1}+\frac{a_2'}{q_2}-t|\lesssim   \frac1{2^lNQ}.$$
		Write $q_1=dm_1$, $q_2=dm_2$. Since we assume $Q\ll N$, it follows that
		$$|{(a_1-a_1')m_2+(a_2-a_2')m_1}|\lesssim \frac{Q^2}{D}|\frac{(a_1-a_1')m_2+(a_2-a_2')m_1}{dm_1m_2}|\lesssim \frac{Q^2}{NDQ2^l}\ll 1.$$
		This forces $a_1q_2+a_2q_1=a_1'q_2+a_2'q_1$, showing that $p$ is an invariant.
		\smallskip
		
		Since $f|p$, there are only $\les 1$ choices for $f$, so $f$ is essentially an invariant. It does become a genuine invariant if $D\gg Q/2^l$ (repeat the above argument for $b_i$), but this will not play any role in our argument.

	\end{re}
	
	\begin{lem}
		\label{l7}
		
		Each $x,t$ has
		\begin{equation}
			\label{admissible}
			\les \min \{P,F+\frac{QP}{DF2^l}\}+\frac{Q^3}{2^lND^2P}\end{equation} admissible pairs.
	\end{lem}
	Before proving this result, let us test its sharpness with two examples.
	\begin{ex}
		When $D=1$, it follows that $P=F=1$, so the upper bound  is $1+Q^3/2^lN$. The proof of \eqref{admissible} is fairly easy in this case, as neither the parameters $p,f$ nor \eqref{e15despreb} play any role. The pair $(q_1,q_2)$ determines $a_1,a_2$ uniquely via \eqref{e15}. Thus, instead of counting pairs $(q_1,q_2)$, we may count four-tuples $(q_1,q_2,a_1,a_2)$ satisfying \eqref{e15}. An argument similar to the one in Proposition \ref{averagevalues} shows that there are (in average) $\sim Q^4/2^lNQ$ solutions for the typical $x,t$. Lemma \ref{l7} confirms that, in the interesting case when $Q\ggg \sqrt{N}$,  this average value is in fact an upper bound for each $x,t$.
	\end{ex}
	
	Next, we analyze an interesting example for the case $D,P,F\sim \sqrt{Q}$.
	\begin{ex}
		\label{exsquareroot}	
		Let $t=\frac{c_2}{r_2}-\frac{c_1}{r_1}$, where $r_1,r_2$ are distinct primes $\sim \sqrt{Q}$ and $1\le c_1\le r_1-1$, $1\le c_2\le r_2-1$. Let also $x=\frac{d_2}{r_2}-\frac{d_1}{r_1}$ where $|d_1|\le r_1-1$, $|d_2|\le r_2-1$.
		
		Pick any prime $r_3\sim \sqrt{Q}$ different from $r_1,r_2$. We prove that $q_1=r_1r_3$, $q_2=r_2r_3$ is an admissible pair for $x,t$ with respect to parameters $D=P=F=\sqrt{Q}$. Indeed, pick any $1\le |c_3|,|d_3|\le r_3-1$. Define $a_1,a_2,b_1,b_2$ via $$\frac{a_1}{q_1}=\frac{c_3}{r_3}-\frac{c_1}{r_1},$$
		$$\frac{a_2}{q_2}=\frac{c_2}{r_2}-\frac{c_3}{r_3},$$
		$$\frac{b_1}{q_1}=\frac{d_3}{r_3}-\frac{d_1}{r_1},$$
		$$\frac{b_2}{q_2}=\frac{d_2}{r_2}-\frac{d_3}{r_3}.$$
		It is easy to see that $\gcd(a_1,q_1)=\gcd(a_2,q_2)=1$, $\frac{a_1}{q_1}+\frac{a_2}{q_2}=t$, $\frac{b_1}{q_1}+\frac{b_2}{q_2}=x$ and that $d=p=f=q_3$. The upper bound in \eqref{admissible} becomes $\les \sqrt{Q}+\frac{Q^{3/2}}{N2^{l}}\lesssim \sqrt{Q}$. On the other hand, the computations above  show that there are at least $\sqrt{Q}/\log Q$ admissible pairs.
	\end{ex}
	\begin{proof}[Proof of Lemma \ref{l7}]
		We may assume that there is at least one admissible pair $(q_1,q_2)$, otherwise we are done. Call $d=\gcd(q_1,q_2)$, so $q_1=dm_1$, $q_2=dm_2$ with $\gcd(m_1,m_2)=1$ and $d\sim D$.
		Let also $a_1,a_2,b_1,b_2,p,f$ satisfy \eqref{e15}-\eqref{e15despref}. These quantities will be fixed, and will serve as a reference point throughout the argument.
		\\
		\\
		Case 1.  Assume  that we can make a choice of $q_1,q_2,a_1,a_2$ such that $\frac{a_1}{q_1}+\frac{a_2}{q_2}\not =0$. Then for any other choice  $q_3,q_4,a_3,a_4$ satisfying \eqref{e15} we must also have $\frac{a_3}{q_3}+\frac{a_4}{q_4}\not =0$. Otherwise, the fact that $|\frac{a_1}{q_1}+\frac{a_2}{q_2}|\gtrsim 1/Q^2$ leads to the contradiction $1/NQ2^l\lesssim 1/Q^2$.
		
		Assume further that, for some $q_1,q_2$ as above,  we can make a choice for $b_1,b_2$ such that $\frac{b_1}{q_1}+\frac{b_2}{q_2}\not =0$.

		Denote the integers
		$$\alpha=\frac{a_1m_2+a_2m_1}{p}$$
		$$\beta=\frac{dm_1m_2}{p}$$
		and recall that they are relatively prime, and $\alpha, \beta\not=0$. Moreover, 
		$\frac{a_1}{q_1}+\frac{a_2}{q_2}=\frac{\alpha}{\beta}$.

		Consider another admissible pair $q_3=d'm_3,q_4=d'm_4$ with $\gcd(m_3,m_4)=1$, and call $a_3,a_4,b_3,b_4,p',f'$ the associated quantities for which \eqref{e15}-\eqref{e15despref} hold true. We have
		\begin{equation}
			\label{ewkjwehfrheu}
			\frac{a_1}{q_1}+\frac{a_2}{q_2}=\frac{a_3}{q_3}+\frac{a_4}{q_4}+ O(\frac1{2^lNQ}).
		\end{equation}
		
		The common denominator for the expression $\frac{a_1}{q_1}+\frac{a_2}{q_2}-\frac{a_3}{q_3}-\frac{a_4}{q_4}$ is $\frac{d}pm_1m_2\frac{d'}{p'}m_3m_4\sim\frac{Q^4}{D^2P^2}$, so we may rewrite \eqref{ewkjwehfrheu} as
		
		\begin{equation}
        \label{f7yf78yrf7yrf7yre7fyr7e8fy7re} 	
		\frac{\alpha}{\beta}=\frac{\frac{a_3m_4+a_4m_3}{p'}}{\frac{d'}{p'}m_3m_4}+\frac{X}{\beta \frac{d'}{p'}m_3m_4}.
		\end{equation}
		
		Here $X$ is an integer. Combining \eqref{ewkjwehfrheu} with \eqref{f7yf78yrf7yrf7yre7fyr7e8fy7re}, it follows that  $$|X|\lesssim \frac{{\beta \frac{d'}{p'}m_3m_4}}{2^lNQ} \sim \frac{Q^3}{D^2P^2N2^l}.$$ We will not use the fact that $X$ is small, only the fact that $X$ has a small number of possible values, namely $1+O(\frac{Q^3}{D^2P^2N2^l})$.

		We find that
		$$\alpha(\frac{d'}{p'}m_3m_4)\equiv X\pmod \beta$$
		or, with $\alpha^{-1}$ being the inverse mod $\beta$,
		$$\frac{d'}{p'}m_3m_4\equiv \alpha^{-1}X\pmod \beta.$$
		
		Note that $\frac{d'}{p'}m_3m_4\sim \beta$, so
		for each $X$, $\frac{d'}{p'}m_3m_4$ may take $O(1)$ values. Each such value determines $\frac{d'}{p'},m_3,m_4$  up to $\les 1 $ many choices. There are $\sim P$ possibilities for $p'$, thus $\les P$ possible values of $d'$. We conclude that there are $\les P$ admissible pairs $q_3,q_4$ for each $X$.
		
		For the  nonzero values of $X$, we are content with the upper bound $\les \frac{Q^3}{D^2PN2^l}$ for the total number of associated admissible pairs $q_3,q_4$. When $X=0$, there are $\les P$ such pairs.
		\smallskip
		
		We will prove an alternative upper bound to $\les P$ in the case of $X=0$. Since $\gcd(\alpha,\beta)=1=\gcd(\frac{a_3m_4+a_4m_3}{p'},{\frac{d'}{p'}m_3m_4})$, the equality
		$$\frac{\alpha}{\beta}=\frac{\frac{a_3m_4+a_4m_3}{p'}}{\frac{d'}{p'}m_3m_4}$$
		forces $\frac{d'}{p'}m_3m_4=\beta.$ This determines the values $m_3,m_4$ and $d'/p'$ up to $\les 1$ many choices. At this point, we might potentially have as many as $\sim P$ choices for $p'$, leading to the previous upper bound $\les P$ for the number of choices of $d'$. We next show a different estimate, that will work better in our applications.
		Recall that
		\begin{equation}
			\label{jgefefcpp[rt[p]]}
			\frac{b_1}{q_1}+\frac{b_2}{q_2}=\frac{b_3}{q_3}+\frac{b_4}{q_4}+O(\frac1{Q2^l})
		\end{equation}
		and that
		$$\frac{b_1}{q_1}+\frac{b_2}{q_2}=\frac{B}{\frac{p}{f}\beta}$$
		$$\frac{b_3}{q_3}+\frac{b_4}{q_4}=\frac{B'}{\frac{p'}{f'}\beta},$$
		for some integers $B,B'$ with $\gcd(B,p/f)=\gcd(B',p'/f')=1$. Clearing denominators in  \eqref{jgefefcpp[rt[p]]} we find that
		$$B\frac{p'}{f'}=B'\frac{p}{f}+y,$$
		for some $|y|\lesssim \frac{\beta\frac{p}{f}\frac{p'}{f'}}{Q2^l}\sim \frac{QP}{DF^22^l}$. Again, we only care about the fact that the number of possible $y$ is small, not about the size of individual $y$.
		
		Recall that  $\frac{b_1}{q_1}+\frac{b_2}{q_2}\not =0$.
		This is equivalent with $B\not=0$.
		Since $\gcd(B,p/f)=1$ and $p'/f'\sim p/f$, it follows that in the above equation, the value of $y$ determines the value of $p'/f'$ up to $O(1)$ choices. Since there are $\sim F$ possible values of $f'$ and $\lesssim 1+\frac{QP}{DF^22^l}$ values of $y$, we find that there are $O(F+\frac{QP}{DF2^l})$ possible values for $p'$. We conclude that there are $\les F+\frac{QP}{DF2^l}$ admissible pairs $q_3,q_4$ corresponding to $X=0$.
		\begin{re}
			\label{roleoffexplained}	
			It is important to recall the role of the parameter $f'$ in the previous argument for $X=0$. We used the formula $d'=\frac{d'}{p'}\times \frac{p'}{f'}\times f'$ to determine the number of possibilities for $d'$ as follows. First, we used the equation for $t$ to show that there are $\les 1$ choices for $d'/p'$. We then used the equation for $x$ to control the number of possible $p'/f'$. Finally, we used the fact that $f'$ can take $\sim F$ values. This argument explains why we defined $f$ as in \eqref{iojurifuriurugitug} (making $f$ a divisor of $p$), as opposed to $\gcd(\frac{b_1q_2+b_2q_1}{d},d)$ (making $f$ a divisor of $d$). We do not see a way to simplify  the forthcoming argument by working instead with the latter choice for $f$.
		\end{re}
		Case 2. Here we still operate under the assumption that $\frac{a_1}{q_1}+\frac{a_2}{q_2}\not=0$ for all admissible pairs $q_1,q_2$.
		However, we also assume that for each such $q_1,q_2$ we cannot select $b_1,b_2$ such that $\frac{b_1}{q_1}+\frac{b_2}{q_2}\not =0$. 
		
		Using the notation from the previous case, it follows that $B=B'=0$.
		This in turn implies $p'=f'$, proving that there are $O(F)$ choices for $p'$. As in the previous case, we conclude that there are $\les F$ admissible pairs $q_3,q_4$ corresponding to $X=0$.
		\\
		\\
		Case 3. Assume now that $\frac{a_1}{q_1}+\frac{a_2}{q_2}=0$ for all admissible pairs $q_1,q_2$ and all $a_1,a_2$ as in \eqref{e15}.
		This forces $q_1=q_2$ for all admissible pairs. This further implies that $p=d=q_1=q_2$, so this case corresponds to $P=D\sim Q$. We have thus established the first upper bound $P$ for the number of admissible pairs.
		
		Assume also that we can make a selection of $q_1,q_2,b_1,b_2$ such that
		$\frac{b_1}{q_1}+\frac{b_2}{q_2}\not=0$. This will allow us to find a second upper bound for the number of admissible pairs $q_3,q_4$, as follows. Writing $q=q_1=q_2$, $q'=q_3=q_4$ and  $$\frac{b_1}{q}+\frac{b_2}{q}=\frac{B}{q/f}$$ $$\frac{b_3}{q'}+\frac{b_4}{q'}=\frac{B'}{q'/f'}$$ we find that
		$$B\frac{q'}{f'}=B'\frac{q}{f}+y,$$
		for $|y|\lesssim \frac{Q}{F^22^l}$. Since $B\not=0$ and $\gcd(B,q/f)=1$, the value of $y$ determines the value of $q'/f'$. We find an upper bound of $\lesssim F(1+\frac{Q}{F^22^l})$ for the number of admissible pairs of the form  $(q',q')$. This matches \eqref{admissible}, as $P=D$.
		\\
		\\
		Case 4. Assume that for all choices satisfying \eqref{e15}-\eqref{e15despref} we have  $\frac{a_1}{q_1}+\frac{a_2}{q_2}=\frac{b_1}{q_1}+\frac{b_2}{q_2}=0$. The latter condition also forces $F=P$, so in this case we have $F,P,D\sim Q$. The upper bound \eqref{admissible} is trivially satisfied as there are at most $Q$ admissible pairs, as in the previous case.

	\end{proof}

	Lemma \ref{l7} will only be used  with $\min\{P,F+\frac{QP}{DF2^l}\}$ replaced by $F+\frac{QP}{DF2^l}$. We next investigate solutions for the $x$-component.

	\begin{lem}
		\label{l8}
		Let $x,t$ be arbitrary and let $(q_1,q_2)$ be an admissible pair for $x,t$.
		Let $\Bc(q_1,q_2)$ be the collection of all $b_1$ that satisfy \eqref{e15}-\eqref{e15despref} for some $a_1,a_2,b_2,f$ (these parameters are allowed to depend on $b_1$, while according to Remark \ref{invariant}, the value of $p$ is determined by $q_1,q_2,x,t$). Then $\Bc(q_1,q_2)$ has cardinality
		$$\les D+\frac{Q}{2^lF}.$$
	\end{lem}
	\begin{proof}
		Write as before $q_1=dm_1$, $q_2=dm_2$ with $d\sim D$, $\gcd(m_1,m_2)=1$. Note that since
		$$|\frac{b_1}{q_1}+\frac{b_2}{q_2}-x|\lesssim \frac{1}{Q2^l},$$
		the value of $b_1$ determines the value of $b_2$ uniquely. Thus, we are in fact counting pairs $b_1,b_2$. As observed in Remark \ref{invariant}, there are $\les 1$ possible values of $f$. We can thus count the pairs satisfying
		$$\gcd(b_1m_2+b_2m_1,p)=f$$
		for a fixed $f$. We have  $\frac{b_1}{q_2}+\frac{b_2}{q_2}=\frac{B}{\frac{d}{f}m_1m_2}$, for some integer $B$.

		Fix such a pair $(b_1,b_2)$ for reference. 	
		Any other such pair $(b_1',b_2')$ would have to satisfy
		$$\frac{b_1}{q_1}+\frac{b_2}{q_2}=\frac{b_1'}{q_1}+\frac{b_2'}{q_2}+\frac{Y}{\frac{d}{f}m_1m_2}$$
		for some integer $|Y|\lesssim \frac{Q}{DF2^l}$. Note that there are $O(1+\frac{Q}{DF2^l})$ such integers.  For each such $Y$, two solutions would satisfy
		$$\frac{b_1'}{q_1}+\frac{b_2'}{q_2}=\frac{b_1''}{q_1}+\frac{b_2''}{q_2}.$$
		This forces $m_j|b_j'-b_j''$, so there are $\sim D$ many solutions for each $Y$.
		
	\end{proof}	
	
	The extra gain in $F$ in the factor $Q/2^lF$ will not be used in the argument. It is merely recorded for sharpness and future possible use in the range $M\lll N^{1/10}$.
	
	The previous two lemmas combine to allow at most
	\begin{equation}
		\label{jhdhdhudfud}
		(F+\frac{QP}{DF2^l}+\frac{Q^3}{2^lND^2P})(\frac{Q}{2^l}+D)\end{equation}
	fractions $b_1/q_1$ for each $(x,t)$.
	We next discuss the sharpness of this estimate for three families of parameters $(D,P,F)$.
	\begin{re}
		Let us test \eqref{jhdhdhudfud} when $D=1$ (thus also $P=F=1$). In the language of Definition \ref{enumbers} we get that for each box $B$
		$$n_B^*\les (1+\frac{Q^3}{2^lN})(1+\frac{Q}{2^l}).$$
		In the case when $Q\gtrsim \sqrt{N}$ this becomes
		\begin{equation}
			\label{sharppppp}
			n_B^*\les \frac{Q^4}{N2^{2l}}.
		\end{equation}
		This matches the lower bound in Proposition \ref{averagevalues}, proving the sharpness of our upper bounds for $D=1$. As explained earlier, this value of $D$ is the main obstruction for extending our argument beyond the threshold $M\ges N^{1/10}$.

	\end{re}
	
	\begin{re}
		\label{enemiesexplained}	
		If we take $\gcd(a,q)=\gcd(b,r)=1$ in Example \ref{exenemies}, then $D\sim Q$, $P\sim Q/q$ and $F\sim Q/rq$. We have identified $\approx Q^2/rq$ fractions $b_1/q_1$. This number is consistent with the upper bound \eqref{jhdhdhudfud}, as it coincides with $DF$. Moreover, it is easy to check that when $r$ is small, e.g. $r\les 1$, we have that \eqref{jhdhdhudfud} $\sim DF$ for all values of $l$.
	\end{re}
	
	\begin{re}
		The first part of Example \ref{interexam} has $D,P,F\sim \sqrt{Q}$, and there are $\approx Q$ solutions for the typical popular pair. This is again $\sim DF$. Moreover, if $2^l\gtrsim N^{1/3}$ we have  \eqref{jhdhdhudfud} $\sim DF$.
	\end{re}
	
	\begin{de}
		Let $x,t$ be arbitrary real numbers. 
		Fix $N, Q,l$  as before (so $2^lQ\lesssim N$). Fix also $D,P,F$.
		
		Then we let $\Lc(x,t)=\Lc(x,t,N,Q,l,D,P,F)$ be the maximum number of points $(x_{m,1},t_{m,1})$, $(x_{m,2},t_{m,2})\in \R^2$, $1\le m\le \Lc(x,t)$, such that
		
		$$(x_{m,1},t_{m,1})\in [\frac{b_1}{q_1}-\frac{1}{2^lQ},\frac{b_1}{q_1}+\frac{1}{2^lQ}]\times [\frac{a_1}{q_1}-\frac{1}{2^lQN},\frac{a_1}{q_1}+\frac{1}{2^lQN}],$$
		$$(x_{m,2},t_{m,2})\in [\frac{b_2}{q_2}-\frac{1}{2^lQ},\frac{b_2}{q_2}+\frac{1}{2^lQ}]\times [\frac{a_2}{q_2}-\frac{1}{2^lQN},\frac{a_2}{q_2}+\frac{1}{2^lQN}],$$
		for some $q_i=q_{m,i}\sim Q$, and $a_i=a_{m,i},b_i=b_{m,i}$ satisfying $|a_i|,|b_i|\le q_i-1$, $\gcd(a_i,q_i)=1$, $\gcd(q_1,q_2):=d\sim D$,	$$\gcd(\frac{a_1q_2+a_2q_1}{d},\frac{q_1q_2}{d}):=p\sim P,$$$$\gcd(\frac{b_1q_2+b_2q_1}{d},p)\sim F,$$
		$$x_{m,1}+x_{m,2}=x,$$
		$$t_{m,1}+t_{m,2}=t,$$
		and also such that $x_{m,1}$ are $1/N$-separated (or equivalently, $x_{m,2}$ are $1/N$-separated).
	\end{de}

	We get the following corollary of the previous two lemmas.

	\begin{pr}
		\label{p134}	
		For each $x,t$ we have the estimate	
		$$\Lc(x,t)\les (F+\frac{QP}{DF2^l}+\frac{Q^3}{2^lND^2P})(\frac{Q}{2^l}+D)\frac{N}{Q2^l}.$$	
	\end{pr}
	
	\begin{proof}
		For each $m$, let $a_{m,1},a_{m,2},b_{m,1},b_{m,2},q_{m,1},q_{m,2}$  be such that
		$$|\frac{a_{m,j}}{q_{m,j}}-t_{m,j}|\lesssim \frac1{NQ2^l}$$	
		$$|\frac{b_{m,j}}{q_{m,j}}-x_{m,j}|\lesssim \frac1{Q2^l}.$$
		Since $$|\frac{a_{m,1}}{q_{m,1}}+\frac{a_{m,2}}{q_{m,2}}-t|\lesssim \frac1{NQ2^l},$$
		Lemma \ref{l7} shows that there are $\les   F+\frac{QP}{DF2^l}+\frac{Q^3}{2^lND^2P}$ admissible pairs $(q_{m,1},q_{m,2})$. Fix such a pair. Then since
		$$|\frac{b_{m,1}}{q_{m,1}}+\frac{b_{m,2}}{q_{m,2}}-x|\lesssim \frac1{Q2^l},$$
		Lemma \ref{l8} shows that there are $O(D+\frac{Q}{2^l})$ possible values for $b_{m,1}$. Using $|b_{m,1}/q_{m,1}-x_{m,1}|\lesssim 1/Q2^l$ and the $1/N$-separation of $x_{m,1}$, the result follows.

	\end{proof}	
	\bigskip
	\subsection{Brief discussion of different ranges}This is a brief outline of the remaining subsections.
	Recall that we may assume \eqref{fiu3u9fiewld;.} for some fixed $D,P,F$, which we repeat here for easy reference
	\begin{equation}
		\label{fiu3u9fiewld;.easyref}
		\sum_{(v_1,v_2)\in\Pc_{D,P,F}}\Nc_{D,P,F}(v_1,v_2)\ges R^3/K^2.
	\end{equation}
	
	There are three cases. We exploit \eqref{fiu3u9fiewld;.easyref} in two different ways. The counting argument employed in Case 1 uses the fact that there is a pair of vertices sharing $\ges R/K^2$ neighbors.
	
	On the other hand, Cases 2 and 3 use another consequence of \eqref{fiu3u9fiewld;.easyref}, the existence of a fork of size $\ges R/K$.
	\bigskip

	1. The first range we discuss corresponds to $FD\les Q$ and $P\les 2^l$. Heuristically, this is the case of small $D$, in particular it covers the case $D\lesssim \sqrt{Q}$. In this range Proposition \ref{p134} gives the ideal estimate $\Lc(x,t)\les Q$, and working with just one popular pair solves the problem.
	
	This argument requires $M\ges N^{1/10}$.
	It is this case only that enforces this severe restriction  on $M$ with exponent $1/10$. Moreover, this restriction is only needed when $D\les 1$. See Remark \ref{djhjdhjhjh}.
	
	2. We next cover the range of large $P$, namely $P\ggg 2^l$. This case is the most delicate and the most interesting one. We identify a localized piece of the graph we call {\em fork} and find arithmetic structure in it.   This is the only case that uses the hypothesis \eqref{e1again}. We use it for a subset of the original collection of $(x_r,t_r)$, one that is highly structured arithmetically. The needed lower bound is $M\ges N^{1/12}$, slightly more relaxed than $M\ges N^{1/10}$.
	
	A very simple version of this argument is shown to solve the case $D\sim Q$ that is responsible for the ``enemies" described in Example \ref{exenemies}.

	3. The third case is $FD\ges Q$. This is the argument that exposes the role of the parameter $f$. It uses the fork again and some very easy counting. It does not need any lower bound on $M$.

	\begin{re}
		\label{roleoff}	
		Our argument will not use induction on scales (Case 2 above does use the induction hypothesis at a smaller scale, but never the conjectured upper bound for smaller scales). One may fantasize about proving Conjecture \ref{c1} using induction on $Q$, for fixed $M$. The  small values of $Q$ can easily be covered by  Fourier analysis, and would serve as the base case of the induction. There is however a (very limited) range of parameters that could be addressed with this method.
		
		More precisely, assume  $DP\gg Q$, so that
		$$\frac{a_1}{q_1}+\frac{a_2}{q_2}=\frac{A}{q}$$
		with $q\ll Q$ and $(A,q)=1$.
		If it also happens that $f=p$, then we have $\frac{b_1}{q_1}+\frac{b_2}{q_2}=\frac{B}{q}$ for some integer $B$. Assume the graph $G$ is dominated by this scenario, in other words assume that \eqref{fiu3u9fiewld;.easyref} holds with some $D,P,F$ satisfying  $F=P$ and $DP\gg Q$. Then all popular pairs $v_1,v_2$ have their differences $v_2-v_1$ inside intervals of the form  $$[\frac{B}{q}-\frac{1}{Q2^l},\frac{B}{q}+\frac{1}{Q2^l}]\times [\frac{A}{q}-\frac{1}{NQ2^l},\frac{A}{q}+\frac{1}{NQ2^l}].$$
		Pigeonholing, there is $Q_1\ll Q$ such that a significant fraction of the popular pairs are inside such intervals with $q\sim Q_1$. All these intervals are inside $S_{Q_1,l_1}$ where $Q2^l=Q_12^{l_1}$. Let $K_1=2^{l_1/2}/M^2$. Note that $K_1\ge K$. It can be shown (\cite{Ru}) that there are at least $\gtrsim R^2/K$ popular pairs. This forces
		$$ R^2/K\les \sum_{v_2,v_2'\in V(G)}1_{S_{Q_1,l_1}}(v_2-v_2'),$$
		and thus
		$$ R^2/K_1\les \sum_{v_2,v_2'\in V(G)}1_{S_{Q_1,l_1}}(v_2-v_2').$$
		Using the induction hypothesis at scale $Q_1$ we find that
		$$R=|V(G)|\les N/M^4,$$
		as desired.
		
		Induction on scales as described here seems limited to this small range of parameters. Example \ref{exenemies} shows that $F$ can be much smaller than $P$ (cf. Remark \ref{enemiesexplained}, when $r\gg 1$).

	\end{re}
	
	\bigskip
	
	\subsection{The  argument  in the case $FD\les Q$ and $P\les 2^l$}
	\label{sub3}
	
	We want to prove the estimate $R\les N/M^4$ assuming $FD\les Q$ and $P\les 2^l$.
	
	We use the existence of one popular pair $(v_1,v_2)$, one satisfying $\frac{R}{K^2}\les \Nc_{D,P,F}(v_1,v_2)$. We also use Proposition \ref{p134} with $(x,t)=v_1-v_2$ to get that
	$$ \Nc_{D,P,F}(v_1,v_2)\le \Lc(x,t)\les (F+\frac{QP}{DF2^l}+\frac{Q^3}{N2^lD^2P})(\frac{Q}{2^l}+D)\frac{N}{Q2^l}.$$
	
	This gives the desired inequality
	
	\begin{equation}
	\label{fifiufiurfiurifior}	
	R\les \frac{2^l}{M^4}\frac{N}{Q2^l}(F+\frac{QP}{DF2^l}+\frac{Q^3}{N2^lD^2P})(\frac{Q}{2^l}+D)\les \frac{N}{M^4},
	\end{equation}
	once we prove the following result.
	\begin{lem}
		\label{t55nnjhfjh}	
		Assume $M\ges N^{1/12}$. Then
		$$\frac{Q^3}{N2^lD^2P}D\les Q.$$
		Assume $M\ges N^{1/10}$.	Then
		$$
		\frac{Q^3}{N2^lD^2P}\frac{Q}{2^l}\les Q.
		$$
		Assume $M\ges N^{1/12}$, $FD\les Q$ and $P\les 2^l$. Then
		$$(F+\frac{QP}{DF2^l})(\frac{Q}{2^l}+D)\les Q.$$
	\end{lem}
	\begin{proof}	
		Note first that $\frac{Q^2}{N2^lDP}\le \frac{Q^2}{N2^l}\lesssim \frac{N}{2^{3l}}\les 1$ when $M\ges N^{\frac1{12}}$.
		\\
		\\
		Second, $\frac{Q^4}{N2^{2l}D^2P}\le \frac{Q^4}{N2^{2l}}\les Q$. The second inequality follows by combining $Q\lesssim N/2^l$ with $2^l\ges M^4\ges N^{2/5}$. The restriction $M\ges N^{\frac1{10}}$ is needed here, as $D$ (and $P$) could be as small as $1$.
		\\
		\\
		Third,  note that $Q\les 2^{2l}$ if $M\ges N^{1/12}$.
		Using this together with $FD\les Q$ and $P\les 2^l$ we find
		$$(F+\frac{QP}{DF2^l})(\frac{Q}{2^l}+D)\le 2Q(\frac{P}{2^l}+\frac{Q}{2^{2l}})+FD\les Q.$$
	\end{proof}
	
	\begin{re}
		\label{djhjdhjhjh}
		The proof of the lemma shows that, when it comes to verifying \eqref{fifiufiurfiurifior},  the severe lower bound $1/10$ in the exponent for $M$ is only needed in the case $D\les 1$ (as described in the second paragraph of the proof). If we knew that $D$ is larger, a more relaxed restriction would suffice. 
		
		It also shows that if $Q^3\les N2^{2l}$, then the weaker lower bound $M\ges N^{1/12}$ suffices.
	\end{re}

	\bigskip

	\subsection{The  argument  for the case $P\ggg 2^l$}
	\label{4.4}
	
	We will prove the estimate $R\les N/M^4$ assuming $P\ggg 2^l$, in the slightly more relaxed setting $M\ges N^{1/12}$.  We have in particular
	$$DP\ggg 2^{2l}\ggg Q.$$

	The arguments in this section cover the situation when
	$$
	(F+\frac{QP}{DF2^l}+\frac{Q^3}{2^lND^2P})(\frac{Q}{2^l}+D)\sim \frac{QP}{DF2^l}\times D.$$
	This scenario is implied by e.g. the choice of parameters $F\les 1$ and $P\ggg 2^l$. If a graph with such parameters exists (we do not have an example), the earlier counting argument falls apart as $\frac{QP}{DF2^l}\times D\ggg Q$. An inspection of the proofs of Lemmas \ref{l7} and \ref{l8} reveals that this hypothetical graph exhibits some rigidity. Most notably,  \eqref{ewkjwehfrheu} becomes a perfect equality
	$$
	\frac{a_1}{q_1}+\frac{a_2}{q_2}=\frac{a_3}{q_3}+\frac{a_4}{q_4}.$$
	We do not see a way to exploit this rigidity in order to simplify the argument in this subsection.

	\smallskip

	We start with an easy, very instructive version of the argument that proves the bound $R\les N/M^{4}$ when $D=Q$ (the values of $P,F$ are irrelevant in this case). This argument shows how to deal with the enemies in Example \ref{exenemies}.
	
	We recall our working hypothesis, with $D=Q$
	$$
	\sum_{(v_1,v_2)\in\Pc_{D,P,F}}\Nc_{D,P,F}(v_1,v_2)\ges R^3/K^2.
	$$
	Note that
	$$
	\sum_{(v_1,v_2)\in\Pc_{D,P,F}}\Nc_{D,P,F}(v_1,v_2)=\sum_{(v_1,v_3)\in E(G)}\sum_{v_2:\;(v_1,v_2)\in\Pc_{D,P,F}\atop{v_3\in \Nc_{D,P,F}(v_1,v_2)}}1\le\sum_{(v_1,v_3)\in E(G)}\sum_{v_2:\atop{\;v_3\in \Nc_{D,P,F}(v_1,v_2)}}1.
	$$
	Combining it with \eqref{djjfvji} we find a pair $(v_1,v_3)$ such that
	\begin{equation}
		\label{djjirjirjfior}
		|S=\{v_2:\;v_3\in \Nc_{D,P,F}(v_1,v_2) \}|\ges R/K.
	\end{equation}
	We call $S$ (and its later refinements) a {\em fork}.  There is a (fixed) edge between $v_1,v_3$ (the handle of the fork), and there are edges (the tines) connecting $v_3$ to each element $v_2$ of  $S$
	
	Let $a_1,b_1,q_1$ be the unique numbers such that $(v_1-v_3)\blacksquare (a_1,b_1,q_1)$. For each $v_2$ in the set $S$ from above we must have, for some $a_2,b_2$, but with the same $q_1$ (since $D=Q$, all denominators can be assumed to be  equal)
	$$(v_2-v_3)\blacksquare(a_2,b_2,q_1).$$
	It follows that $$S-v_3\subset \bigcup_{|a|,|b|\le q_1-1}[\frac{b}{q_1}-\frac{1}{Q2^l},\frac{b}{q_1}-\frac{1}{Q2^l}]\times[\frac{a}{q_1}-\frac{1}{NQ2^l},\frac{a}{q_1}-\frac{1}{NQ2^l}].$$
	Thus the $x$ coordinates of $S-v_3$ lie in the set$$\bigcup_{|b|\le q_1-1}[\frac{b}{q_1}-\frac{1}{Q2^l},\frac{b}{q_1}-\frac{1}{Q2^l}],$$
	and since they are $1/N$-separated we find that
	$$|S|=|S-v_3|\lesssim Q\times \frac{N}{2^lQ}=\frac{N}{2^l}.$$
	Thus $$\frac{R}{K}\les \frac{N}{2^l},$$
	giving the desired $R\les N/M^4$, in fact even the superficially stronger $R\les N/2^{l/2}M^2$.
	
	This argument uses the transitivity of the relation $\gcd(q_1,q_2)\sim Q$  to show that the fork contains only one denominator $q_1$. While this is not true when $D\ll Q$, further refinements of the fork will still make crucial use of the small cardinality of denominators they contain.
	\medskip
	
	Let us now move to the general case $P\ggg 2^l$. Fix $(v_1,v_3)$ such that \eqref{djjirjirjfior} holds. Assume as before that $(v_3-v_1)\blacksquare (a_1,b_1,q_1)$. For each $v_2\in S$ we have
	$(v_2-v_3)\blacksquare (a_2,b_2,q_2)$, with $\gcd(q_1,q_2)\sim D$. Since $q_1$ is fixed, there are $\les 1$ possibilities for the value of $\gcd(q_1,q_2)$. We may thus find a $d\sim D$ such that the refined fork $S'$ satisfies
	$$|S'=\{v_2:\;v_3\in \Nc_{D,P,F}(v_1,v_2),\; (v_2-v_3)\blacksquare (a_2,b_2,q_2),\;\gcd(q_1,q_2)=d \}|\ges R/K.$$
	Since this $d$ has $\les 1$ divisors $p$, we may find  $p\sim P$ such that the refined fork $S''$
	$$S''=$$$$\{v_2:\;v_3\in \Nc_{D,P,F}(v_1,v_2),\; (v_2-v_3)\blacksquare (a_2,b_2,q_2),\;\gcd(q_1,q_2)=d,\;\gcd(\frac{a_1q_2+a_2q_1}{d},\frac{q_1q_2}{d})=p \}$$
	satisfies
	$$|S''|\ges R/K.$$
	\smallskip
	
	The next lemma uncovers critical structure in  $S''$.
	\begin{lem}
		\label{structure}	
		Let  $v_2,v_2'\in S''$. Assume $(v_2-v_3)\blacksquare(a_2,b_2,q_2)$ and $(v_2'-v_3)\blacksquare (a_2',b_2',q_2')$ with $q_2=dm_2$, $q_2'=dm_2'$. Then
		$$\frac{a_2}{q_2}-\frac{a_2'}{q_2'}=\frac{a}{q}$$
		for some divisor $q$ of $\frac{d}{p}m_2m_2'$, and some $a$ relatively prime to $q$.
	\end{lem}
	\begin{proof}
		Write $q_1=dm_1$. On the one hand we trivially have
		$$\frac{a_2}{q_2}-\frac{a_2'}{q_2'}=\frac{A_1}{dm_2m_2'}.$$
		On the other hand,
		$$\frac{a_2}{q_2}-\frac{a_2'}{q_2'}=(\frac{a_2}{q_2}+\frac{a_1}{q_1})-(\frac{a_2'}{q_2'}+\frac{a_1}{q_1})=\frac{A_2}{\frac{d}{p}m_1m_2}-\frac{A_2'}{\frac{d}{p}m_1m_2'}=\frac{A}{\frac{d}{p}m_1m_2m_2'}.$$
		Here $A_1,A$ are integers.
		We find that
		$$A_1m_1=Ap.$$
		If $A_1=0$ we are done, taking e.g. $a=0$, $q=1$. So we may assume $A,A_1\not=0$. Recalling that $\gcd(p,m_1)=1$ (Lemma \ref{primenessofp}), we find that $p|A_1$. We may take $q=\frac{dm_2m_2'}{\gcd(dm_2m_2',A_1)}$.
		
	\end{proof}

	We next recall (see \eqref{e1again}) that for each $v_2=(x,t)\in S''$ we have
	$$|\sum_{n\in\Z}a_nw(\frac{n}{N})e(nx+n^2t)|\ge MN^{1/4}.$$
	We use this  hypothesis at the smaller scale of $|S''|$.
	Repeating the argument in Section \ref{sec2}, with the refinement in Remark \ref{dyad},
	we find that there must exist $Q_1$, $l_1$ with $Q_12^{l_1}\le N$ such that, with $K_1=\frac{2^{l_1/2}}{M^2}$ we have
	\begin{equation}
		\label{ririgiguggogtpogprth[ytp}
		|S''|^2/K_1\les \sum_{v_2,v_2'\in S''}1_{S^{dyad}_{Q_1,l_1}}(v_2-v_2').
	\end{equation}
	The previous lemma shows that for each $v_2,v_2'\in S''$, writing $v_2-v_2'=(x,t)$, we have that
	\begin{equation}
		\label{cdjdhruerifmc[ir]}
		|t-\frac{a}{q}|\lesssim \frac1{NQ2^l},
	\end{equation}
	for some $q\lesssim \frac{Q^2}{DP}$ and $a\in\Ac(q)$. Since we also have $DP\gtrsim Q$, we find that $q\lesssim Q$. Thus, we also have
	$$|t-\frac{a}{q}|\lesssim \frac{1}{Nq}.$$
	On the other hand, picking some (any!)  $v_2,v_2'\in S''$ with $v_2-v_2'\in S_{Q_1,l_1}^{dyad}$ and writing again $v_2-v_2'=(x,t)$, we must have\begin{equation}
		\label{cdjdhruerifmc[ir2]}
		|t-\frac{a'}{q'}|\lesssim \frac{1}{Nq'},
	\end{equation}
	for some $q'\sim Q_1$, $\gcd(a',q')=1$.

	Since $2^{l_1}, 2^l\ges M^4$ it follows that $Q, Q_1\ll N$. This shows that the  major arcs \eqref{cdjdhruerifmc[ir]} and \eqref{cdjdhruerifmc[ir2]} can only intersect if $q=q'$ (and also $a=a'$). Thus, $Q_1\lesssim Q^2/DP$.

	Also, since in fact
	$$|t-\frac{a'}{q'}|\begin{cases}\sim \frac{1}{NQ_12^{l_1}},\;\text{if }2^{l_1}<N/Q_1\\\lesssim 1/N^2,\;\text{if }2^{l_1}=N/Q_1\end{cases}$$
	and $a'/q'=a/q$, combining this with \eqref{cdjdhruerifmc[ir]} shows that
	$\frac{1}{NQ_12^{l_1}}\lesssim \frac{1}{NQ2^l}$, and thus
	\begin{equation}
		\label{critlow}
		2^{l_1-l}\gtrsim Q/Q_1\gtrsim DP/Q\ggg 1.
	\end{equation}
	\smallskip
	
	In the next proposition we seek an upper bound for $\sum_{v_2,v_2'\in S''}1_{S^{dyad}_{Q_1,l_1}}(v_2-v_2')$, that will eventually be used in conjunction with the lower bound \eqref{ririgiguggogtpogprth[ytp}.
	\begin{pr}
		For each $v_2\in S''$ we have
		$$|\{v_2'\in S'': v_2-v_2'\in S_{Q_1,l_1}^{dyad}\}|\les \frac{N}{2^{l_1}}.$$
	\end{pr}
	\begin{proof}
		Assume as before $(v_2-v_3)\blacksquare(a_2,b_2,q_2)$ and $(v_2'-v_3)\blacksquare (a_2',b_2',q_2')$ with $q_2=dm_2$, $q_2'=dm_2'$.
		Lemma \ref{structure} shows that $a_2m_2'-a_2'm_2=pa'$, so that we have
		$$\frac{a_2}{q_2}-\frac{a_2'}{q_2'}=\frac{a'}{\frac{d}{p}m_2m_2'}.$$
		We next analyze ways in which this fraction may potentially be simplified to $a/q$ with $\gcd(a,q)=1$. In particular, we are interested in the size of $q$.
		
		Write $m=\gcd(m_2,m_2')$, $m_2=n_2m$, $m_2'=n_2'm$. Recall that $\gcd(m,p)=1$, so $m$ must divide $a'$. Writing  $a'=a''m$ we may first simplify the above fraction to $\frac{a''}{\frac{d}{p}mn_2n_2'}$. Let now $r=\gcd(a'', \frac{d}{p}mn_2n_2')$.  It follows that $q=\frac{dm}{pr}n_2n_2'=\frac{dm_2m_2'}{prm}$.
		Since we know that $q\sim Q_1$, we find that
		\begin{equation}
			\label{efo0rif0iig9ig9i}
			Q_1\sim \frac{Q^2}{DPmr}.
		\end{equation}
		
		We now move the analysis to the $x$ component. Call $x_2,x_2'$ the $x$-components of $v_2,v_2'$. We have from hypothesis that
		$$|(x_2-x_2')-(\frac{b_2}{q_2}-\frac{b_2'}{q_2'})|\lesssim \frac1{Q2^l}.$$ Since $v_2-v_2'\in S_{Q_1,l_1}^{dyad}$, there must exist $|b|\le q-1$ such that $|(x_2-x_2')-\frac{b}q|\lesssim \frac1{Q_12^{l_1}}$. Thus
		$$|\frac{b_2'}{dm_2'}-\frac{b_2}{dm_2}-\frac{b}{q}|\lesssim \frac1{Q2^l}+\frac1{Q_12^{l_1}}\sim \frac1{Q2^l},\;\;\;\text{by }\eqref{critlow}. $$
		Recalling that $q$ divides $dm_2m_2'$, we may write the expression inside the absolute value signs in the form $\frac{c}{dm_2m_2'}$. This forces $c=0$, with the immediate consequence
		\begin{equation}
			\label{rejficif0og-oyv56,o-}
			|(x_2-x_2')-(\frac{b_2}{q_2}-\frac{b_2'}{q_2'})|\lesssim \frac1{Q_12^{l_1}}.
		\end{equation}
		
		Indeed, if not, the previous inequality implies that $\frac1{dm_2m_2'}\lesssim \frac1{Q2^l}$, or $D2^l\lesssim Q$. However, this is impossible under our assumptions $P\ggg 2^l$ and $M\ges N^{1/12}$.
		
		We may thus write
		$$\frac{b_2'n_2-b_2n_2'}{dmn_2n_2'}=\frac{b}{q}=\frac{prb}{dmn_2n_2'}.$$
		In particular,
		\begin{equation}
			\label{eorjfrfprforfroegvhoyph}
			pr|b_2'n_2-b_2n_2'.\end{equation}
		
		We note that $\gcd(r,n_2)=1$. This follows from the fact that  $r$ divides $a_2n_2'-a_2'n_2$ and  $\gcd(a_2,n_2)=\gcd(n_2',n_2)=1$. Similarly, $\gcd(r,n_2')=1$. Thus, $r$ is a divisor of $\frac{dm}{p}$, in particular it is a divisor of the fixed quantity $\frac{dm_2}{p}$. Also $m|m_2$. We conclude that there are $\les 1$ possible values for the parameters $m,r$.
		\smallskip
		
		We are finally ready to count the number of $v_2'$. Fix $m,r$. We focus attention on those $v_2'$ for which $q_2'=dm_2'$ satisfies $\gcd(m_2,m_2')=m$. There are $O(Q/Dm)$ such $q_2'$. We fix $q_2'$
		and count the possible values of $b_2'$, such that $(v_2'-v_3)\blacksquare (a_2',b_2',q_2')$ for some $a_2'$. While each $b_2'$ may be associated with a different (but unique) $a_2'$, we may further restrict attention to those $b_2'$ for which the associated $a_2'$ satisfies
		$$\frac{a_2}{q_2}-\frac{a_2'}{q_2'}=\frac{a}{q}$$
		with $q=\frac{dm_2m_2'}{prm}$ and $\gcd(a,q)=1$. As observed above, this restriction forces \eqref{eorjfrfprforfroegvhoyph}. Recall that $n_2,n_2',b_2$ are fixed at this point, so if $b_2''/q_2'$ is another fraction in our restricted category, we must have
		$$pr|n_2(b_2''-b_2').$$
		Since $\gcd(p,m_2)=1$ (Lemma \ref{primenessofp}), we have $\gcd(p,n_2)=1$. Thus, $\gcd(pr,n_2)=1$, and $pr$ must divide $b_2''-b_2'$. This shows that we have $O(Q/pr)$ possible values of $b_2'$ subject to the previous restriction.
		
		We have proved that there are $\lesssim \frac{Q}{Dm}\frac{Q}{Pr}$ fractions $b_2'/q_2'$ associated with the fixed parameters $m,r$. For each such $b_2'/q_2'$, there could be only $O(N/Q_12^{l_1})$ values of $v_2'$ with $(v_2'-v_3)\blacksquare (a_2',b_2',q_2')$ and $v_2-v_2'\in S_{Q_1,l_1}^{dyad}$. This follows from \eqref{rejficif0og-oyv56,o-} and the $1/N$-separation of the points $x_2$. Invoking \eqref{efo0rif0iig9ig9i} concludes  that the number of such $v_2'$ is
		$$\lesssim \frac{Q^2}{DPmr}\frac{NDPmr}{Q^22^{l_1}}=\frac{N}{2^{l_1}}.$$
		To close the argument, we use the previous observation that $m,r$ may take $\les 1$ possible values.

	\end{proof}

	We now combine the previous proposition with \eqref{ririgiguggogtpogprth[ytp}
	and find $|S''|\les K_1\frac{N}{2^{l_1}}$. Since $|S''|\ges R/K$, and using \eqref{critlow} this gives
	\begin{equation}
		\label{foro5-oy-p.-7p=n887in=-}
		R\les 2^{\frac{l-l_1}{2}}\frac{N}{M^4}\lll \frac{N}{M^4}.
	\end{equation}

	This shows that  sharp examples (those for which $R\approx  N/M^4$) do not exist if  $DP\ggg Q$. This should be contrasted with Examples \ref{exemplu1} and \ref{interexam} featuring  graphs  with $DP\sim Q$.

	\bigskip
	
	\begin{re}
		The main point of Lemma \ref{structure} is to lead to the improved estimate $Q_1\lesssim Q^2/DP$ over the trivial estimate $Q_1\lesssim Q^2/D$. This improved estimate plays a critical role for us, as it forces  the  lower bound $2^{l_1-l}\gtrsim 1$ (see \eqref{critlow}) crucially used in \eqref{foro5-oy-p.-7p=n887in=-}. The trivial estimate would only give the useless lower bound $2^{l_1-l}\gtrsim D/Q$.
	\end{re}

	\subsection{Proof in the case $FD\ges Q$}
	\label{sub7}
	
	 It is the argument in this section that will make critical use of the parameter $f$. Our argument does not work under the weaker assumption $PD\ges Q$ (this information is not strong enough), see Remark \ref{notenough} below.
	\bigskip

	We further refine the fork $S''$ introduced in Subsection \ref{4.4} by writing, for some fixed $f$
	$$S'''=\{v_2:\;v_3\in \Nc_{D,P,F}(v_1,v_2),\; (v_2-v_3)\blacksquare (a_2,b_2,q_2) \text{ such that}$$$$\;\gcd(q_1,q_2)=d,\;\gcd(\frac{a_1q_2+a_2q_1}{d},\frac{q_1q_2}{d})=p,\; \gcd(\frac{b_1q_2+b_2q_1}{d},p)=f\}.$$
	Recall that $v_1,v_3$ are fixed such that $(v_1-v_3)\blacksquare (a_1,b_1,q_1)$.
	Recall that $p\sim P$ and $d\sim D$ are also fixed, but $a_2,b_2,q_2$ are arbitrary. Since $f|p$, we may pick $f\sim F$ such that
	$$|S'''|\ges R/K.$$
	We next find an upper bound for $|S'''|$. Due to $1/N$-separation, for each fixed $b_2,q_2$ there are $O(\frac{N}{Q2^l})$ possible $v_2$ satisfying $(v_2-v_3)\blacksquare (a_2,b_2,q_2)$ for some arbitrary $a_2$. It remains to count the possible pairs $(q_2,b_2)$. As $q_2=dm_2$ is determined by $m_2$, there are $Q/D$ choices for $q_2$. Let us now fix $q_2=dm_2$, and count how many $b_2$ may be paired with $q_2$. We use the key property that
	$$f|b_1m_2+b_2m_1.$$
	Recalling that $\gcd(p,m_1)=1$, it follows that $\gcd(f,m_1)=1$. Thus, any other $b_2'$ that is paired with $q_2$, since it also satisfies $f|b_1m_2+b_2'm_1,$ would be forced to satisfy $f|b_2-b_2'$. We find that there can only be $q_2/f\sim Q/F$ many $b_2$ that are paired with $q_2$.
	
	These lead to the upper bound
	$$|S'''|\lesssim \frac{N}{Q2^l}\frac{Q}{D}\frac{Q}{F}.$$
	Our critical assumption $FD\ges Q$ in this section  implies that $|S'''|\les N/2^l$. When combined with the lower bound $R/K$, we find
	\begin{equation}
		\label{betterandnoM}
		R\les \frac{N}{M^22^{l/2}}.
	\end{equation}
	
	\begin{re}
		\label{notenough}	
		The parameter $F$ allows us to count the entries $b_2$ paired with $q_2$, which eventually  leads to a good upper bound for $|S'''|$. For comparison, the parameter $P$ provides an upper bound on the number of $a_2$ that are paired with a given $q_2$. Indeed, if
		$$\gcd(\frac{a_1q_2+a_2q_1}{d},\frac{q_1q_2}{d})=\gcd(\frac{a_1q_2+a_2'q_1}{d},\frac{q_1q_2}{d})=p,$$
		then we must have in particular $p|(a_1m_2+a_2m_1)-(a_1m_2+a_2'm_1)=m_1(a_2-a_2')$. Since $\gcd(p,m_1)=1$, we find the  upper bound $Q/P$ for the number of $a_2$ associated with $q_2$. Thus, we have an upper bound of $Q^2/DP$ for the number of possible pairs $a_2,q_2$ in the relation $(v_2-v_3)\blacksquare (a_2,b_2,q_2)$.
		While $Q^2/QD\lesssim Q$, we learn nothing from this about the number of possibilities for $b_2$, and thus nothing about the size of $S'''$. Recalling that each $v_2$ is of the form $(x,t)$, we have found an upper bound on the possible choices for $t$. But since each $t$ can be associated with many $x$, we do not get a sensible bound on the number of these $x$. On the other hand, the successful argument in this section counts directly the possible entries for $x$, using implicitly that each $x$  is associated with only one $t$ (the entries $b_2,q_2$ determine the entry $a_2$ up to $O(N/2^lQ)$ possibilities).

	\end{re}
	
	\section{Recovering the $N^{1/3}$ bound in $L^6$ and beyond}
	\label{sub9}
	A (much) simpler version of the earlier argument  recovers the sharp $N^{1/3}$ bound in $L^6$
	$$\|\sup_t|\sum_{n\in\Z}a_nw(\frac{n}{N})e(nx+n^2t)|\|_{L^6([0,1])}\les N^{1/3}.$$

	Assuming again that for some $M\le N^{1/4}$ and $1\le r\le R$
	$$|\sum_{n\in\Z}a_nw(\frac{n}{N})e(nx_r+n^2t_r)|\ge MN^{1/4},$$
	we need to prove that
	$$R\les (N/M^4)^{3/2}.$$
	It suffices to prove this for $M\ggg N^{1/12},$ otherwise the inequality is trivial. There are only two cases.
	\\
	\\
	Case 1. Assume $FD\lesssim Q$. We use the bound $$R\les (F+\frac{QP}{DF2^l}+\frac{Q^3}{2^lND^2P})(\frac{Q}{2^l}+D)\frac{N}{QM^4},$$
	that follows by using Proposition \ref{p134} and the fact that there is at least one popular pair.
	We analyze the six terms. First,
	
	$$\frac{Q^3}{2^lND^2P}(\frac{Q}{2^l}+D)\frac{N}{QM^4}\lesssim \frac{Q^3}{2^{2l}M^4}+\frac{Q^2}{2^lM^4}\lesssim \frac{N^3}{M^{24}}+\frac{N^2}{M^{16}}\lesssim \frac{N^{3/2}}{M^6},$$
	the last inequality being true precisely when $M\gtrsim N^{1/12}$.
	Then, using that $FD\lesssim Q$ and thus also $F\lesssim \sqrt{Q}$
	$$F(\frac{Q}{2^l}+D)\frac{N}{QM^4}\lesssim \frac{\sqrt{Q}}{2^l}\frac{N}{M^4}+\frac{N}{M^4}\lesssim \frac{N}{M^4}\lesssim (\frac{N}{M^4})^{3/2}.$$
	We used that $Q\lesssim 2^{2l}$ precisely when $M\gtrsim N^{1/12}$.
	Then, since $P\le D$ we have
	$$\frac{QP}{DF2^l}\frac{Q}{2^l}\frac{N}{QM^4}\lesssim \frac{Q}{2^{2l}} \frac{N}{M^4}\lesssim (\frac{N}{M^4})^{3/2}.$$
	Finally, using the crude bound $P\lesssim Q$
	$$\frac{QP}{DF2^l}D\frac{N}{QM^4}\lesssim \frac{Q}{2^{l}} \frac{N}{M^4}\lesssim \frac{N^2}{M^{12}}\lesssim (\frac{N}{M^4})^{3/2},$$
	with the last inequality true precisely when $M\gtrsim N^{1/12}$.
	\\
	\\
	Case 2. If $FD\gtrsim Q$, recall that we got the bound $R\lesssim N/M^{4}$ in \eqref{betterandnoM}, for all values of $M\ge 1$, in particular for $M\gtrsim N^{1/12}$. This bound is better than the needed $(N/M^4)^{3/2}$ since $M\le N^{1/4}$.

	\subsection{A potential strategy for breaking the $M\ges N^{1/10}$ barrier}
	We describe a plausible strategy  to prove sharp level set estimates \eqref{e7600} for  $M$  slightly smaller than $N^{1/10}$, by reducing things to the purely number theoretic Conjecture \ref{conditional}. In particular, we put flesh on the intuition recorded in Remark \ref{random}. The new idea is to consider joint neighbors of three vertices, rather than two. We do not quantify the gain and do not aim for optimizing the argument,  as the approach is conditional to Conjecture \ref{conditional}. We show heuristics that suggest that, if executed efficiently, the three-neighbor argument should extend the range to $M\ggg N^{1/12}$. Extra leverage may further be obtained by considering joint neighbors of more than three vertices.

	Assume $M$ is sufficiently close to $N^{1/10}$, by a margin determined at the end of the argument. At the very least, we assume $M\ggg N^{1/12}$. We find it convenient to do a proof by contradiction. Assume we have a configuration for which
	\begin{equation}
		\label{qw1}
		R\ggg N/M^4.\end{equation}
	This allows us to assume that
	\begin{equation}
		\label{qw2}
		Q^3\ggg N2^{2l},
	\end{equation}	
		otherwise Remark \ref{djhjdhjhjh} is applicable.
	
	In particular,
	\begin{equation}
		\label{qw4}
		2^l\lll N^{2/5}.
	\end{equation}
	
	Finally,  Remark \ref{djhjdhjhjh}  also allows us to assume that
	the graph is dominated by some $D=N^{\alpha}$, with $\alpha>0$ as small as we wish, subject to Conjecture \ref{conditional}. In other words, we may assume that for some $F\le P\le D=N^\alpha$ we have
	$$
	\sum_{(v_1,v_2)\in\Pc_{D,P,F}}\Nc_{D,P,F}(v_1,v_2)\ges R^3/K^2.
	$$
	As we have seen earlier, this implies the existence of  a collection $S$ of distinct
	\begin{equation}
		\label{qw3}
		L\ges R/K
	\end{equation}
	vertices $v_1,v_2,\ldots,v_{L}$ such that
	$(v_1,v_i)\in\Pc_{D,P,F}$ for each $2\le i\le L$.
	Let $\Nc_i$ be a subset of $\Nc_{D,P,F}(v_1,v_i)$ with $|\Nc_i|\approx R/K^2$. We have that
	$$ \frac1R\|\sum_{i=2}^L1_{\Nc_i}\|_1^2\le\|\sum_{i=2}^L1_{\Nc_i}\|_2^2=\|\sum_{i=2}^L1_{\Nc_i}\|_1+\sum_{i\not=j}|\Nc_i\cap\Nc_{j}|.$$
	We may assume that \begin{equation}
	\label{jijigjiogjioj}
	K\lll R^{1/3},
	\end{equation}
	 as the opposite inequality combined with \eqref{qw4} contradicts \eqref{qw1}, as follows
	$$R\les K^3\sim \frac{2^{3l/2}}{M^6}\lll\frac{N^{3/5}}{M^6}\lll \frac{N}{M^4}.$$
	Now \eqref{jijigjiogjioj} implies that
	$$\|\sum_{i=2}^L1_{\Nc_i}\|_1\ges R^2/K^3\ggg R,$$
	leading to
	$$\frac1{R}(LR/K^2)^2\les \frac1R\|\sum_{i=2}^L1_{\Nc_i}\|_1^2\lesssim \sum_{2\le i\not=j\le L}|\Nc_i\cap\Nc_{j}|.$$
	As $|\Nc_i\cap\Nc_{j}|\les R/K^2$, it follows that for at least  $\ges L^2/K^2$ pairs $i\not=j$ we have $|\Nc_i\cap\Nc_{j}|\ges R/K^4$. In particular, there is some $2\le i\le L$, say $i=2$, such that $|\Nc_2\cap\Nc_{j}|\ges R/K^4$ for at least $\ges L/K^2$ values $j$.
	\smallskip
	
	We can be more precise. Write $v_j=(x_j,t_j)$. We may think of two distinct $t_j$ as being $1/N^2$-separated. At most $O(N^{1/2}/M^2)$ of the vertices in $S$ may share a fixed second entry $t$, see  Remark \ref{duiweyfucfi90gi}. Thus there are $\lesssim \frac{N^{3/2}}{M^2Q2^l}$ vertices $v_j$ with second entry $t_j$ within $\lesssim 1/NQ2^l$ from $t_2$.
	
	Note that, using first \eqref{qw3}, then \eqref{qw1}, then \eqref{qw2} and the fact that $M\ge N^{1/24}$ we find
	$$\frac{L}{K^2}\ges \frac{R}{K^3}\ggg \frac{NM^2}{2^{3l/2}}\ggg \frac{N^{3/2}}{M^2Q2^l}.$$

	It follows that there are $V\ges L/K^2$
	values of $j$ with $t_{j}$ being $\gg 1/NQ2^l$ away from $t_2$  and such that
	$|\Nc_2\cap\Nc_{j}|\ges R/K^4$. For our argument it suffices to pick one such $j$, say $j=3$. However, we need to make sure that $V$ is at least 1.  
		Using \eqref{qw3} and  \eqref{jijigjiogjioj} we get
	$$V\ges \frac{L}{K^2}\ges \frac{R}{K^3}\ggg 1.$$
	This together with \eqref{transitiviti} shows that $V\ggg 1$. We record our findings below
	\begin{equation}
	\label{hfcrfycrft767t687t76y}
	|\Nc_2\cap\Nc_{3}|\ges R/K^4\text{ and }|t_3-t_2|\gg 1/NQ2^l.
	\end{equation}
	\smallskip

	Lemma \ref{l7} shows that there are $\les N^\alpha+ Q^3/2^lN$  solutions $(a_1,a_2,q_1,q_2)$ with $\gcd(a_i,q_i)=1$ and $\gcd(q_1,q_2)\sim N^\alpha$ for the inequality
	$$|\frac{a_1}{q_1}-\frac{a_2}{q_2}-(t_2-t_1)|\lesssim \frac1{NQ2^l}.$$
	\smallskip

	Let us now see what happens if we factor  the vertex $v_3$ into the analysis.
	If $v\in \Nc_2\cap \Nc_3$ satisfies  $(v-v_1)\blacksquare(a_1,b_1,q_1)$, $(v-v_2)\blacksquare(a_2,b_2,q_2)$,
	$(v-v_3)\blacksquare(a_3,b_3, q_3)$
	then we have
	\begin{equation}
		\label{uffinal1}
		|\frac{a_1}{q_1}-\frac{a_2}{q_2}-(t_2-t_1)|\lesssim \frac1{NQ2^l}\end{equation}
	and \begin{equation}
		\label{uffinal2}
		|\frac{a_1}{q_1}-\frac{a_3}{q_3}-(t_3-t_1)|\lesssim \frac1{NQ2^l}.
	\end{equation}

	The following conjecture seems plausible. Since we aim for improvements for $M$ only slightly smaller than $N^{1/10}$ (and certainly the argument needs $M\ggg N^{1/12}$, which implies $Q\lesssim N/2^l\les N/M^4\lll N^{2/3}$), we may restrict attention to  $Q\le N^{2/3}$. This conjecture introduces the saving $N^\beta$ over the number of solutions for just one equation.
	\begin{con}
		\label{conditional}	
		There is $\beta>0$ such that for all small enough $\alpha>0$ the following is true.
		
		Assume that $Q2^l\lesssim N$ with $2^l,Q\ge 1$ and $Q\le N^{2/3}$. For each $t,t'$ with $|t-t'|\gg 1/NQ2^l$, the number of solutions $(a_1,a_2,a_3,q_1,q_2,q_3)$ with  $q_i\sim Q$, $|a_i|<q_i$, $\gcd(a_i,q_i)=1$, $\gcd(q_1,q_2),$ $\gcd(q_1,q_3)\le N^\alpha$, for the system of inequalities
		$$|\frac{a_1}{q_1}-\frac{a_2}{q_2}-t|\lesssim \frac1{NQ2^l}$$
		and
		$$|\frac{a_1}{q_1}-\frac{a_3}{q_3}-t'|\lesssim \frac1{NQ2^l}$$
		is $\les N^\alpha+\frac{Q^3}{2^lN^{1+\beta}}$.
	\end{con}
	It is necessary to ask that $|t-t'|\gg 1/NQ2^l$, otherwise the two inequalities coincide, and, as observed earlier, the number of solutions may be as large as $\les N^\alpha+\frac{Q^3}{2^lN}$ (just consider $a_3=a_2,q_3=q_2$). We believe this separation condition is enough to compensate for the lack of any restriction on $\gcd(q_2,q_3)$, which our previous argument cannot enforce.
	
	The reason  why Conjecture \ref{conditional} holds for generic $t,t'$ is that there are $\sim Q^6$ tuples $(a_1,\ldots,q_3)$ and $\sim (NQ2^l)^2$ pairwise disjoint squares with side length $1/NQ2^l$.  Thus, there are $\lesssim Q^4/N^22^{2l}$ solutions for  generic $t,t'$. This is in fact $\les 1$ when $M\ges N^{1/12}$.
	\medskip

	If true, this conjecture proves that there are  $\les N^\alpha+Q^{3}/2^lN^{1+\beta}$ solutions $(a_1,\ldots,q_3)$ to the system \eqref{uffinal1}, \eqref{uffinal2}. This is an improvement over the earlier bound on the number of admissible pairs from Lemma \ref{l7}. The conjecture has the following (conditional) corollary.
	\begin{co}
	\label{oiofirgiurtguiiotuhiou}
	There is $\beta>0$ such that for all small enough $\alpha>0$ the following is true.
		
		Assume that $Q2^l\lesssim N$ with $2^l,Q\ge 1$ and $Q\le N^{2/3}$. For each $t,t'$ with $|t-t'|\gg 1/NQ2^l$, and for arbitrary $x,x'$,  the number of solutions $(a_1,a_2,a_3,b_1,b_2,b_3,q_1,q_2,q_3)$ with  $q_i\sim Q$, $|a_i|,|b_i|<q_i$, $\gcd(a_i,q_i)=1$, $\gcd(q_1,q_2),$ $\gcd(q_1,q_3)\le N^\alpha$, for the system of inequalities
		$$|\frac{a_1}{q_1}-\frac{a_2}{q_2}-t|\lesssim \frac1{NQ2^l},$$
		$$|\frac{a_1}{q_1}-\frac{a_3}{q_3}-t'|\lesssim \frac1{NQ2^l},$$
		$$|\frac{b_1}{q_1}-\frac{b_2}{q_2}-x|\lesssim \frac1{Q2^l},$$
		$$|\frac{b_1}{q_1}-\frac{b_3}{q_3}-x'|\lesssim \frac1{Q2^l}$$
		is $\les (N^\alpha+\frac{Q^3}{2^lN^{1+\beta}})(N^\alpha+\frac{Q}{2^l})$.
	\end{co}	
	\begin{proof}
	Using Conjecture \ref{conditional}, the first two equations determine $(a_1,a_2,a_3,q_1,q_2,q_3)$ up to $\les N^\alpha+\frac{Q^3}{2^lN^{1+\beta}}$ many choices. Let us fix such a six-tuple. Using Lemma \ref{l8} and the third equation determines the value of $b_1$ up to $\les N^\alpha+\frac{Q}{2^l}$ many choices. Finally, note that the third equation (together with $b_1$) determines $b_2$ uniquely, while the fourth equation determines $b_3$ uniquely.  
	
	\end{proof}

	Note that $\frac{Q}{2^l}\ggg N^{1/5}$, as a consequence of \eqref{qw2} and \eqref{qw4}.  Thus $Q/2^l$ dominates $N^\alpha$ if $\alpha\le 1/5$. We repeat the argument in Proposition \ref{p134}, applying Corollary \ref{oiofirgiurtguiiotuhiou} with $(x,t)=v_2-v_1$, $(x',t')=v_3-v_1$, and using \eqref{hfcrfycrft767t687t76y}. We get the bound 
	$$R/K^4\les |\Nc_2\cap\Nc_3|\les \frac{N}{Q2^l}(N^\alpha+\frac{Q^3}{2^lN^{1+\beta}})\frac{Q}{2^l}.$$
	This in turn proves that
	$$R\les \frac{N^{1+\alpha}K^4}{2^{2l}}+\frac{Q^3}{N^\beta M^82^l}\lesssim  \frac{N^{1+\alpha}}{M^8}+\frac{N^3}{N^{\beta}2^{4l}M^{8}}\les \frac{N}{M^4},$$
	with the last inequality being true as long as $M\ges N^{\frac{2-\beta}{20}}$ and $\alpha$ is small enough. This contradicts \eqref{qw1}.
	
	The argument described here works even if Conjecture \ref{conditional} is false for a small number of exceptions. Conjecture \ref{conditional} seems difficult even in the extreme case $\alpha=0$, and even when $\gcd(q_1,q_2)= \gcd(q_1,q_3)=\gcd(q_2,q_3)=1.$

\end{document}